\documentclass{imsart}
\RequirePackage[OT1]{fontenc}
\RequirePackage{amsthm,amsmath}
\RequirePackage[numbers]{natbib}
\RequirePackage[colorlinks,citecolor=blue,urlcolor=blue]{hyperref}
\RequirePackage{hypernat}
\usepackage{amssymb,graphicx}
\advance \topmargin by -\headheight \advance \topmargin by
-\headsep

\def\a{\alpha}
\def\b{\beta}

\def\m{\mu}
\def\s{\sigma}

\def\l{\lambda}
\def\D{{\cal D}}

\def\H{{\mathbb H}}
\def\R{\mathbb R}
\def\IK{I\!\!K}
\def\JK{J\!\!K}

\def\F{{\mathbb F}}

\def\H{{\mathbb H}}
\def\argmax{\mbox{argmax}}
\def\conv{\mathop{\rm \longrightarrow}}
\newcommand{\Dconv}{\mbox{$\displaystyle \conv^{\cal D}$}}

\startlocaldefs
\numberwithin{equation}{section}
\theoremstyle{plain}
\newtheorem{thm}{Theorem}[section]
\newtheorem{corollary}{Corollary}[section]
\newtheorem{lemma}{Lemma}[section]
\newtheorem{remark}{Remark}[section]

\endlocaldefs

\begin{document}

\begin{frontmatter}
\title{Smooth and non-smooth estimates of a monotone hazard\protect\thanksref{T1}}
\runtitle{Monotone hazard estimation}
\thankstext{T1}{We thank Jon Wellner for the fruitful cooperation we had for many years!}

\begin{aug}
\author{\fnms{Piet Groeneboom}\ead[label=e1]{P.Groeneboom@tudelft.nl}}
\and
\author{\fnms{Geurt Jongbloed}\ead[label=e2]{G.Jongbloed@tudelft.nl}}

\runauthor{P.\ Groeneboom and G.\ Jongbloed}

\affiliation{Delft University of Technology}

\address{Delft Institute of Applied Mathematics\\
Mekelweg 4, 2628 CD Delft\\
The Netherlands\\
\printead{e1}\\
\phantom{E-mail:\ }
\printead*{e2}}
\end{aug}

\begin{abstract}
We discuss a number of estimates of the hazard under the assumption that the hazard is monotone on an interval $[0,a]$. The usual isotonic least squares estimators of the hazard are inconsistent at the boundary points $0$ and $a$. We use penalization to obtain uniformly consistent estimators. Moreover, we determine the optimal penalization constants, extending related work in this direction by \cite{woodroofe_sun:93} and \cite{woodroofe_sun:99}. Two methods of obtaining smooth monotone estimates based on a non-smooth monotone estimator are discussed. One is based on kernel smoothing, the other on penalization.
\end{abstract}



\end{frontmatter}

\section{Introduction}
\label{section:intro}
In survival analysis and reliability theory, the hazard rate (also known as failure rate) is a natural function to model the distribution of data. It describes the probability of instantaneous failure at time $x$, given the subject has functioned until $x$. The exponential distributions are the only distributions with constant hazard rate, which is related to the `memoryless property' of this distribution. Other shapes of the hazard rate indicate whether the object suffers ageing (increasing hazard rate) or is getting more reliable having survived longer (decreasing hazard rate).

In estimating a hazard function under the restriction that it is monotone, popular methods are maximum likelihood and isotonic least squares projection (\cite{Robertson:88}, section 7.4). These estimators are typically piecewise constant and non-smooth. More recently, the method of monotonic rearrangements was studied in \cite{neumeyer07}. Depending on the choice of the initial estimator, these estimators can be smooth. Methods to obtain smooth estimators of the hazard rate include plug-in ratio estimators and smoothed empirical hazards as discussed in \cite{ricerosen76}. See  \cite{singpurwallawong:83} for an overview of the various estimators. These estimators are typically not monotone. In \cite{GrJo10a} the so-called maximum smoothed likelihood estimator was introduced, an estimator that is both smooth and monotone. In this paper, a type of non-smooth as well as smooth monotone estimators of a  monotone hazard rate will be studied. Before giving an outline of the paper, some words on our motivation to study this problem.

The problem of testing a null hypothesis of exponentiality (constant hazard rate) against the alternative of a monotone hazard rate, was extensively studied in the sixties of the preceding century; see e.g.\ \cite{prospyke:67}. Only quite recently, the problem of testing the null hypothesis of monotonicity of a hazard rate has received attention.  \cite{gh04} consider a multiscale version of the Proshan-Pyke test, and compute critical values based on the exponential distribution.   \cite{hallk:05} use an integral type test statistic that is based on second order differences of the empirical cumulative hazard function and approximate the critical values of the test using bootstrap samples from a well chosen smoothed version of the empirical cumulative hazard function.   \cite{durot:08} studies the supremum distance between two estimators of the cumulative hazard and obtains critical values using the exponential distribution. An alternative approach to this testing problem is developed in  \cite{GrJo10c}. There  an integral-type test statistic is introduced and a bootstrap approach is used  to determine approximate critical values.  This approach is shown to be less conservative than methods based on the exponential distribution and less anticonservative than the method proposed in \cite{hallk:05}. In order for the bootstrap method described in \cite{GrJo10c} to work well,  estimators for a  locally monotone hazard rate are needed that are smooth and uniformly consistent on the interval of monotonicity and behave properly near the boundary of the interval of monotonicity.

In this paper we concentrate on the nonparametric least squares method to estimate a locally monotone hazard rate and discuss smooth and non-smooth versions of this approach. It is well-known that the ``raw" least-squares or maximum likelihood method yields inconsistent estimates at the boundary (this will also be seen in section \ref{section:isotonic_estimates}). Following an approach introduced in the context of density estimation in \cite{woodroofe_sun:93,woodroofe_sun:99}, we introduce a penalty at the endpoints in the least squares criterion. In Theorem \ref{th:optimal_penalization} in section \ref{section:isotonic_estimates}  the asymptotically optimal penalization constants, minimizing an asymptotic mean squared error criterion, are determined. The optimal order of the penalization constants turns out to be $n^{-2/3}$, if $n$ is the sample size and it is assumed that the hazard is strictly increasing on the interval of interest. Somewhat different recommendations were given in \cite{woodroofe_sun:93,woodroofe_sun:99}, where penalization constants of the order $(\log n)/n$ and $1/\sqrt{n}$ were used, respectively (see also Remark \ref{remark:flat_hazards}).

There are several methods that can be used to construct smooth estimators based on a basic non-smooth monotone estimator discussed in section \ref{section:isotonic_estimates}. One method that automatically leads to monotone estimators, is kernel smoothing. In section \ref{section:kernel_estimates} this method  is described and the resulting estimator is shown to be asymptotically normally distributed. Moreover,  both locally and globally optimal bandwidths are determined for estimating the hazard rate.

In section \ref{section:penalization_estimates} smooth estimates based on penalizing the estimates of section \ref{section:isotonic_estimates} are studied. The penalization uses an integral over the square of the derivative of the hazard, as used in  \cite{tantwood:94} and \cite{palwood:07}. We show that full minimization of the penalized criterion yields a uniformly consistent estimate of the hazard, but gives inconsistent estimates of the derivative of the hazard at the boundary points, since the derivatives tend to zero at the boundary, as in \cite{tantwood:94} and \cite{palwood:07}. We remedy the latter difficulty by introducing two boundary conditions  in order to get consistent estimates of the derivative of the hazard, also at the boundary points. Having consistent estimates of the derivative of the hazard is important in generating bootstrap samples for finding critical values of (isotonic) tests for monotone hazards in the setting of \cite{GrJo10c}.

\section{Monotone least-squares estimates of the hazard}
\label{section:isotonic_estimates}
\setcounter{equation}{0}
Suppose we have a sample $X_1,\dots,X_n$ from a distribution function $F_0$ on $[0,\infty)$, with density $f_0$ and hazard function $h_0$. This latter function characterizes the distribution function $F_0$, which can be seen by the relation
$$
h_0(x)=-\frac{d}{dx}\log(1-F_0(x))=\frac{f_0(x)}{1-F_0(x)}
$$
with inverse
$$
F_0(x)=1-\exp\left(-\int_0^x h_0(y)\,dy\right).
$$
If one wants to estimate the hazard $h_0$ under the restriction that it is monotone on the interval $[0,a]$, one of the simplest estimates is the least squares estimate $\hat h_n$, which minimizes the quadratic criterion
\begin{equation}
\label{quad_criterion1}
\tfrac12\int_0^a h(x)^2\,dx-\int_{[0,a]} h(x)\,d\H_n(x),
\end{equation}
under the restriction that $h$ is monotone. Here $\H_n$ is the empirical cumulative hazard function
$$
\H_n(x)=-\log\left\{1-\F_n(x)\right\},\,x<\max_i X_i,
$$
and $\F_n$ is the empirical distribution function of the sample $X_1,\dots,X_n$. The rationale behind this criterion function is that $\H_n$ will be close to $H_0$ (defined as $\int_0^xh_0(y)\,dy$) asymptotically and $h\mapsto \tfrac12\int_0^ah(x)^2\,dx-\int_0^ah(x)\,dH_0(x)$ is minimized by taking $h=h_0$ (which can be seen by `completing the square'). Another option is to use maximum likelihood methods, but in view of our restriction of the monotonicity hypothesis to an interval, this method has more complications in the present case, so we will concentrate on least squares methods in this paper. For specificity, we shall consider the hypothesis that $h$ is nondecreasing on $[0,a]$, although similar methods can be used if the hypothesis is that $h$ is nonincreasing on $[0,a]$ or monotone on a compact interval not including zero.

The solution of the problem of minimizing (\ref{quad_criterion1}) is well-known, and found in the following way. Construct the so-called {\it cusum diagram}, consisting of the point $(0,0)$, and the points
$$
\left(X_{(i)},\H_n(X_{(i)}-)\right),\, 1\le i\le n,\,X_{(i)}<a,\qquad \left(a,\H_n(a-)\right),
$$
where the $X_{(i)}$ are the order statistics of the sample, and where we assume $X_{(n)}>a$. Then the solution $\hat h_n$ of the minimization problem is given by the left-continuous derivative of the greatest convex minorant of this cusum diagram.

To illustrate the behavior of the estimators in this paper, we introduce the family of hazards $\{h^{(d)}\,:\,d\in[-1,1]\}$, also considered in \cite{hallk:05}.
The corresponding distribution functions on $(0,\infty)$ are given by
\begin{equation}
\label{F_d}
F^{(d)}(x)=1-\exp\left\{-\tfrac12x-\tfrac52\left\{\tfrac14\left(x-\tfrac34\right)^4+\left(\tfrac34\right)^3x\right\}-\tfrac13dx^3+\tfrac58\left(\tfrac34\right)^4\right\}.
\end{equation}
If $d>0$ we get a strictly increasing hazard; if $d<0$, the hazard is decreasing on
$$
\left(\frac34-\frac2{15}d-\frac{2}{15}\sqrt{d^2-\frac{45}{4}d},\frac34-\frac2{15}d+\frac{2}{15}\sqrt{d^2-\frac{45}{4}d}\right)
$$
and if $d=0$ the hazard has a stationary point at $x=3/4$. See Figure \ref{fig:dfamilyhaz} for some hazards and corresponding densities in this family.

\begin{figure}
\begin{center}
\includegraphics[width=6cm]{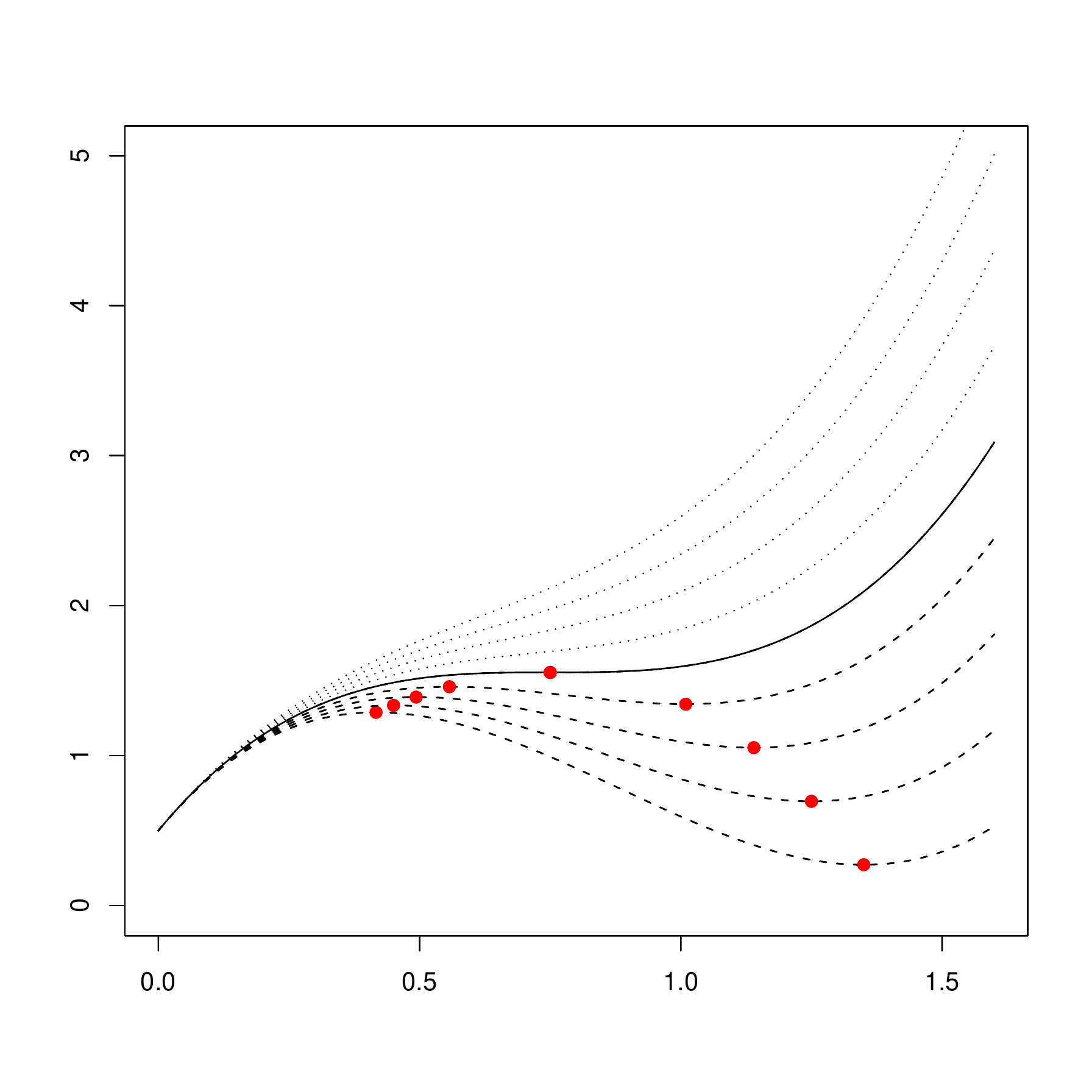}
\includegraphics[width=6cm]{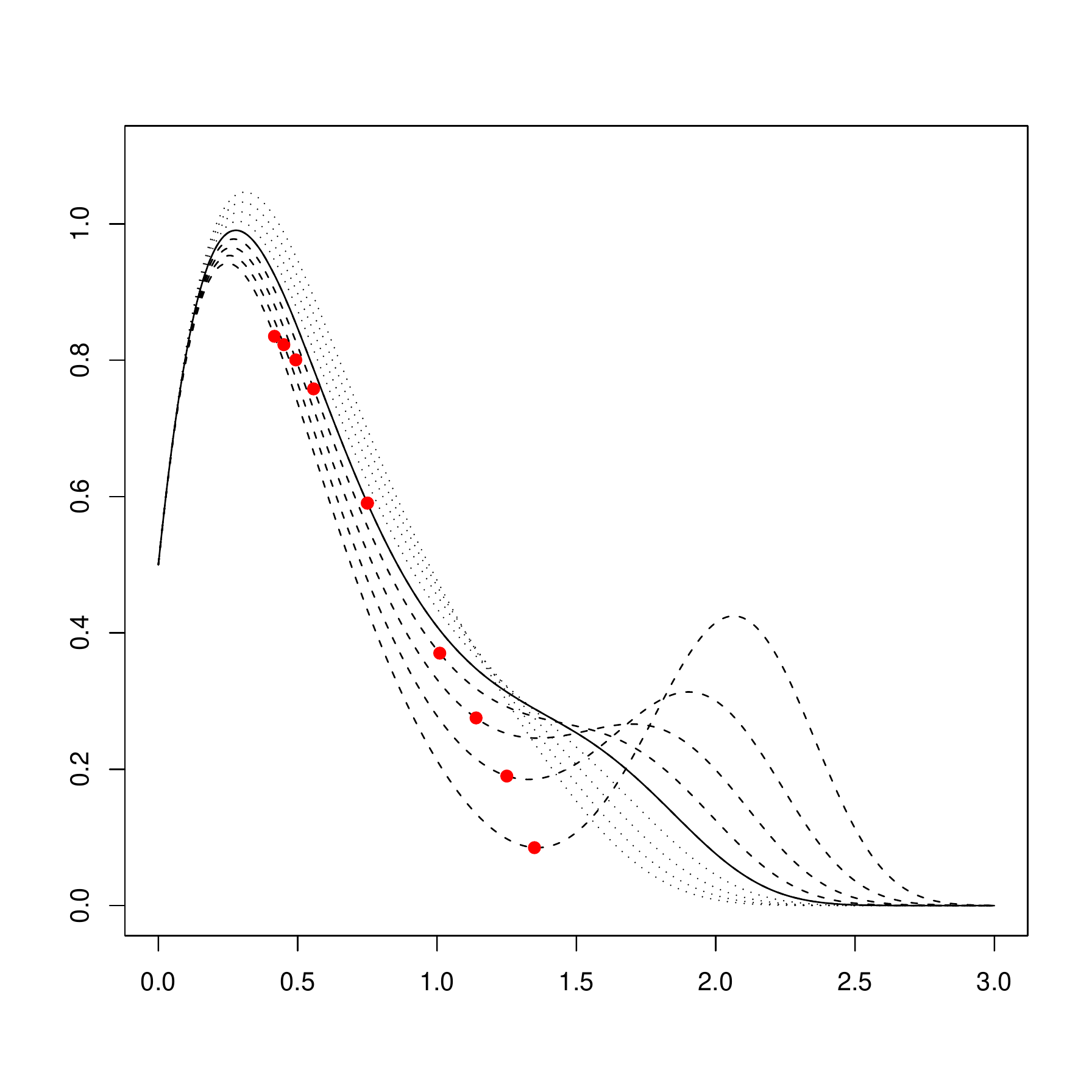}
\end{center}
\caption{The left panel shows the hazard functions $h^{(d)}$ for $d=-1,-0.75,-0.50,-0.25$ (dashed), $d=0$ (full curve) and $d=0.25,0.50,0.75,1$ (dotted) corresponding to distribution functions (\ref{F_d}). The stationary points are shown by the red dots. The right panel shows the corresponding densities.}
\label{fig:dfamilyhaz}
\end{figure}

\begin{remark}
{\rm
Note that we need the constant $\tfrac58\left(\tfrac34\right)^4$ in the exponent to make the distribution function zero at the left endpoint $0$, but that this constant is missing in the formula given below (4.1) on p.\ 1121 in \cite{hallk:05}.
}
\end{remark}

A picture of the cusum diagram and its greatest convex minorant (red) for a sample of size $n=100$ from the distribution function $F^{(1)}$ on the interval $[0,\left(F^{(1)}\right)^{-1}(0.95)]$ and the corresponding estimate of the hazard function are shown in Figure \ref{fig:hazard100_unpen}.

\begin{figure}
\begin{center}
\includegraphics[width=6cm]{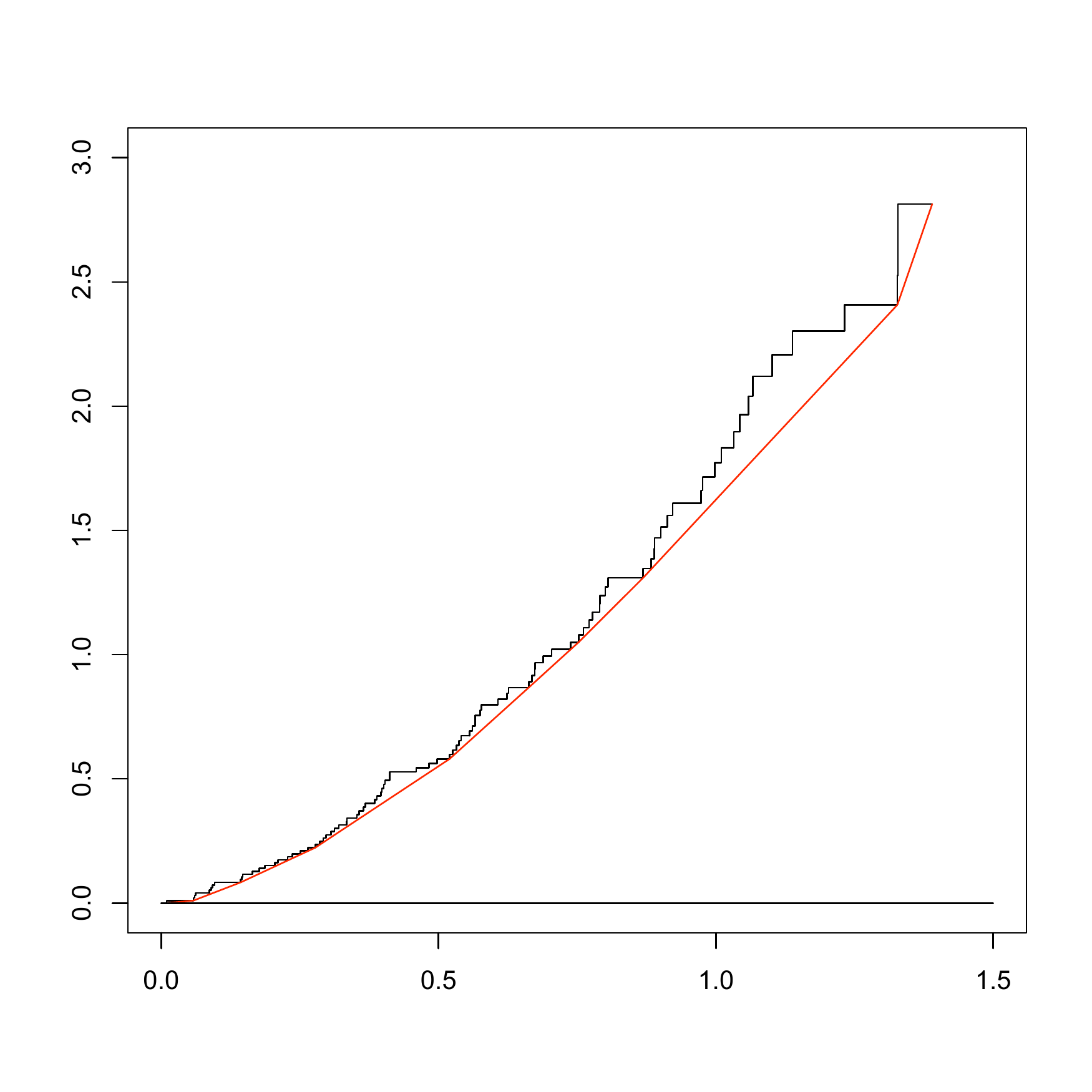}
\includegraphics[width=6cm]{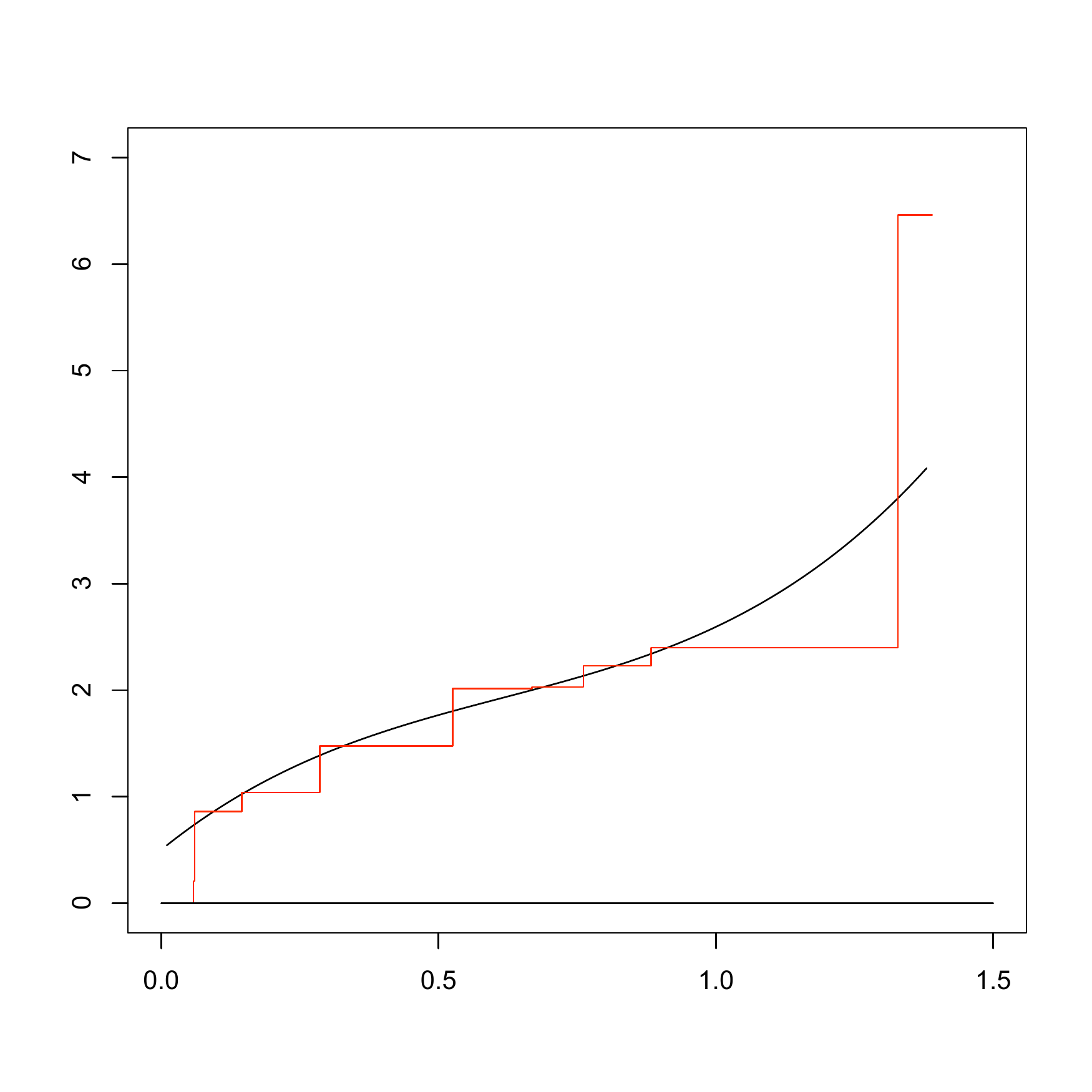}
\end{center}
\caption{The (unpenalized) cusum diagram and its greatest convex minorant (left panel) and the corresponding least squares estimate of the hazard (right panel)  for a sample of size $n=100$ from the distribution function $F^{(1)}$ on the interval $[0,\left(F^{(1)}\right)^{-1}(0.95)]$. The real hazard is the black curve in the right panel.}
\label{fig:hazard100_unpen}
\end{figure}

The lemma below shows that on intervals that stay away from the boundary points $0$ and $a$, the hazard estimator is uniformly consistent.
\begin{lemma}
\label{lem:conshnint}
Suppose $0\le\a_n,\b_n\rightarrow0$ as $n\to\infty$. Let $h_0$ be continuous and nondecreasing on $[0,a]$. Then for each $0<\delta<a/2$,
\begin{equation}
\label{eq:unifconvhn}
\sup_{[\delta,a-\delta]}|\hat{h}_n(x)-h_0(x)|\rightarrow0\mbox{ with probability one}.
\end{equation}
\end{lemma}
\begin{proof} The argument is similar to that in Theorem 3 in \cite{GroJo95}.  First note that  $\H_n$ converges to $H_0$ uniformly on $[0,a]$ almost surely by the Glivenko Cantelli theorem. Since a.s.\ for any $\epsilon>0$, $H_0-\epsilon\le \hat{H}_n\le {\H}_n\le H_0+\epsilon$ on $[0,a]$  for all $n$ sufficiently large (since $\hat H_n$ is {\it the greatest} convex minorant of $\H_n$ and $H_0-\epsilon$ is a.s.\ {\it a} convex minorant of $\H_n$ for $n$ sufficiently large),  $\hat H_n$ converges to $H_0$ uniformly on $[0,a]$ almost surely.

Now fix $x\in(0,a)$. Then for each $\epsilon>0$ such that $(x-\epsilon,x+\epsilon)\subset[a,b]$, we have by definition of $\hat h_n$
$$
\frac{\hat{H}_n(x)-\hat{H}_n(x-\epsilon)}{\epsilon}\le \hat h_n(x)\le \frac{\hat{H}_n(x+\epsilon)-\hat{H}_n(x)}{\epsilon}.
$$
The left hand side converges a.s.\ to $(H_0(x)-H_0(x-\epsilon))/\epsilon$; the right hand side to $(H_0(x+\epsilon)-H_0(x))/\epsilon$. Since $\epsilon$ was chosen arbitrarily, this shows (by continuity of $h_0$ on $[0,a]$) that $\hat h_n(x)\rightarrow h_0(x)$ w.p.\ 1. Uniform convergence on $[\delta,a-\delta]$  follows by monotonicity of both $\hat h_n$ and $h_0$ and continuity of $h_0$ on $[0,a]$. \end{proof}

It is well-known that this estimate has the undesirable feature of being inconsistent at the boundary points $0$ and $a$, and indeed one notices in Figure \ref{fig:hazard100_unpen} that the estimate $\hat h_n(0)$ is too low and the estimate $\hat h_n(a)$ is too high. In fact, it immediately follows from the representation of $\hat{h}_n$ that $\hat{h}_n=0$ on $(0,X_{(1)})$ for all $n$. To remedy a similar problem in the context of maximum likelihood estimation of a monotone density, \cite{woodroofe_sun:93,woodroofe_sun:99} suggest to introduce a penalty at the endpoints. We also use that method in the present situation.

To this end, we introduce the penalized cusum diagram, consisting of the point $(0,0)$, and the points
\begin{equation}
\label{cusum2}
\left(X_{(i)},\H_n(X_{(i)}-)+\a_n\right),\,X_{(i)}<a,\qquad \left(a,\H_n(a-)+\a_n-\b_n\right),
\end{equation}
where $\a_n$ and $\b_n$ are  nonnegative penalty parameters. The left derivative of the present cusum diagram minimizes the criterion
\begin{equation}
\label{LS_criterion2}
\tfrac12\int_0^a h(x)^2\,dx-\int_{[0,a]} h(x)\,d\H_n(x)-\a_n h(0)+\b_n h(a),
\end{equation}
over all nondecreasing functions $h$ on $[0,a]$. Consistency of the resulting estimator on $[\delta,a-\delta]$ is obtained by following the proof of Lemma \ref{lem:conshnint}. This characterization of the estimator  also leads to consistency of $\hat h_n$ at the boundary points $0$ and $a$. Moreover, the optimal order of convergence to zero of the parameters $\alpha_n$ and $\beta_n$ which is important in order to get a feeling for what to do  in practice, can be determined. In \cite{woodroofe_sun:93} it is suggested to take a related penalty of order $(\log n)/n$ and in \cite{woodroofe_sun:99}  to take a penalty of order $1/\sqrt{n}$.
One of the statements in the theorem below is that, under the assumption that $h_0$ stays away from zero and is strictly increasing on $[0,a]$, the optimal penalty is of order $n^{-2/3}$.

\begin{thm}
\label{th:optimal_penalization}
Let $h_0$ be nondecreasing on $[0,a]$ with strictly positive and continuous (one-sided) derivatives at $0$ and $a$. Let $0\le\a_n,\b_n\rightarrow0$.
Then:
\begin{enumerate}
\item[(i)] For each $0<\delta<a$, with probability one, for all $n$ sufficiently large $$\hat{h}_n(0)=\inf_{x\in[0,\delta]}\frac{\H_n(x)+\alpha_n}{x} \mbox{ and } \hat{h}_n(a)=\sup_{x\in[a-\delta,a]}\frac{\H_n(a)-\beta_n-\H_n(x)}{a-x}.
    $$
\item[(ii)] The asymptotically MSE optimal rates for the penalization parameters are $\a_n,\beta_n\sim n^{-2/3}$.
\item[(iii)] Let $W$ be standard Brownian motion on $[0,\infty)$. Taking $\a_n=\a n^{-2/3}$ and $b_n=\b n^{-2/3}$,
\begin{equation}
\label{eq:asympdistzero}
n^{1/3}\left(\hat h_n(0)-h_0(0)\right)\Dconv \inf_{t>0}\left(\frac{W(h_0(0)t)}{t}+\frac{\alpha}{t}+\tfrac12h_0^{\prime}(0)t\right)
\end{equation}
and
$$
n^{1/3}\left(\hat h_n(a)-h_0(a)\right)\Dconv \inf_{t>0}\left(\frac{W(h_0(a)t)}{t}+\frac{\beta}{t}+\tfrac12h_0^{\prime}(a)t\right)
$$
\item[(iv)] The asymptotically MSE-optimal choices for the penalization parameters are $\alpha_n=\alpha n^{-2/3}$ and $\beta_n=\beta n^{-2/3}$ where $\a>0$ is the minimizer of
$$
E\min_{t>0}\left\{\tfrac12h_0'(0)t+\left\{\a+W(h_0(0)t)\right\}/t\right\}^2,
$$
and $\b>0$ is the minimizer of
$$
E\min_{t>0}\left\{\tfrac12h_0'(a)t+\left\{\b+W(h_0(a)t)\right\}/t\right\}^2.
$$
\end{enumerate}
\end{thm}
\begin{proof} We concentrate on the situation at $x=0$ with $\a_n$ as penalty parameter. The right boundary at $x=a$ with penalty parameter $\b_n$ can be dealt with similarly.\\
(i) Fix $\delta>0$. The local assumption on $h_0$ near zero implies that $x\mapsto H_0(x)/x$ is strictly increasing on $(0,a]$. Hence
\begin{equation}
\label{eq:locdelt}
\inf_{x\in [\delta,a]}\frac{H_0(x)}{x}=\frac{H_0(\delta)}{\delta}>\frac{H_0(\delta/2)}{\delta/2}.
\end{equation}
For convenience, write $\tilde{\H}_n=\H_n+\alpha_n1_{(0,\infty)}-\beta_n 1_{[a,\infty)}$. By the uniform convergence of $\tilde{\H}_n$ to $H_0$ on  $[0,a]$, $x\mapsto \tilde{\H}_n(x)/x$ converges to $x\mapsto H_0(x)/x$ uniformly on $[\delta/2,a]$. Combined with (\ref{eq:locdelt}), this shows that with probability one, for all $n$ sufficiently large
$$
\inf_{x\in [\delta,a]}\frac{\tilde{\H}_n(x)}{x}
>\frac{\tilde{\H}_n(\delta/2)}{\delta/2},
\mbox{ implying }\hat{h}_n(0)=\inf_{x\in[0,\delta]}\frac{\H_n(x)+\alpha_n}{x}.
$$ \\
(ii) Let $(\alpha_n)$ be given and  $(\delta_n)$ be a sequence with $0<\delta_n\rightarrow0$ and $n\delta_n\rightarrow\infty$ as $n\rightarrow\infty$. Consider the localized and centered process
$$
t\mapsto \frac{\H_n(\delta_nt)-H_0(\delta_nt)}{\delta_nt}+\frac{\alpha_n}{\delta_nt}+\frac{H_0(\delta_nt)}{\delta_nt}-h_0(0)=V_n(t)
$$
and note that
$$
\hat h_n(0)-h_0(0)=\inf_{t>0}V_n(t).
$$
For fixed $t$,
\begin{equation}
\label{eq:Vn}
V_n(t)=(n\delta_n)^{-1/2}\frac{W_n(t)}{t}+\frac{\alpha_n/\delta_n}{t}+\tfrac12h_0^{\prime}(0)\delta_n t+\tfrac12(h_0^{\prime}(\xi_n)-h_0^{\prime}(0))\delta_n t
\end{equation}
where $0\le\xi_n\le \delta_n t$ and $W_n(t)$ is an asymptotically non-degenerate random variable. Moreover, for $W$ standard Brownian Motion on $[0,\infty)$,
\begin{equation}
\label{eq:convbrown}
W_n(t)=\sqrt{\frac{n}{\delta_n}}\left(\H_n(\delta_nt)-H_0(\delta_nt)\right)\Dconv W(h_0(0)t)
\end{equation}
in $D([0,\infty))$ endowed with the topology of uniform convergence on compacta. Ignoring the (asymptotically negligible) last term in (\ref{eq:Vn}), we see that balancing the two deterministic terms yields $\delta_n\sim \sqrt{\alpha_n}$; taking $\delta_n$ converging either faster or slower to zero than this, will lead to a slower rate of convergence of $V_n$ to zero. Using this choice, $V_n(t)=O_P\left(n^{-1/2}\alpha_n^{-1/4}\right)+O\left(\alpha_n^{1/2}\right)$. This shows that starting off with $\alpha_n\sim n^{-2/3}$ leads to the fastest rate of convergence of $V_n(t)$ to zero. \\
(iii) We now take $\delta_n=n^{-1/3}$ and $\alpha_n=\alpha n^{-2/3}$ with $\alpha>0$. Also using (i) and the local assumption on $h_0$ near zero leads for any $\nu>0$ to the approximate asymptotic representation
\begin{eqnarray*}
n^{1/3}\left(\hat h_n(0)-h_0(0)\right)&=&\inf_{t\in(0,\nu n^{1/3}]}n^{1/3}V_n(t)\\&=&\inf_{t\in(0,\nu n^{1/3}]}\left(\frac{W_n(t)}{t}+\frac{\alpha}{t}+\tfrac12h_0^{\prime}(0)t\right)
\end{eqnarray*}
where ignoring the last term in (\ref{eq:Vn}) is justified because $\nu$ can be chosen arbitrarily small ($0\le \xi_n\le \nu$). In Lemma \ref{lem:locsum} in the appendix we show that by taking $M>0$ sufficiently large and $\epsilon>0$ sufficiently small,
$$
n^{1/3}\left(\hat h_n(0)-h_0(0)\right)=\inf_{t\in(\epsilon,M]}\left(\frac{W_n(t)}{t}+\frac{\alpha}{t}+\tfrac12h_0^{\prime}(0)t\right)
$$
with arbitrarily high probability. Together with (\ref{eq:convbrown}), and the fact that for Brownian Motion on $[0,\infty)$
$$
\inf_{t>0}\left(\frac{W(h_0(0)t)}{t}+\frac{\alpha}{t}+\tfrac12h_0^{\prime}(0)t\right)=\inf_{t\in(\epsilon,M]}\left(\frac{W(h_0(0)t)}{t}+\frac{\alpha}{t}+\tfrac12h_0^{\prime}(0)t\right)
$$
with arbitrarily high probability by taking $\epsilon>0$ sufficiently small and $M>0$ sufficiently large, this leads to (\ref{eq:asympdistzero}). Finally, the optimal asymptotically MSE-optimal value for $\alpha$ in (iv) is obtained by minimizing the expectation of the square of the right hand side of (\ref{eq:asympdistzero}) as a function of $\alpha$.  \end{proof}

\begin{remark}
\label{remark:flat_hazards}
{\rm Theorem \ref{th:optimal_penalization} gives the optimal penalization constants for the case that the hazard is strictly increasing on $[0,a]$. The situation is quite different if, e.g., the hazard is constant on $[0,\delta]$ for some $\delta>0$. In view of (\ref{eq:Vn}), the linear (in $t$) terms are not present, and in order to make $V_n(t)$ as small as possible, $\delta_n$ should {\it not} tend to zero and $\alpha_n$ should be chosen of the order $n^{-1/2}$.
So in this case the type of scaling used in \cite{woodroofe_sun:99} seems the more natural type of scaling.
The limit behavior of the greatest convex minorant which one gets in this case (for $\a=0$) is analyzed in \cite{piet:83}, whereas the limit situation for the case that the greatest convex minorant corresponds to a strictly convex function (where $n^{-2/3}$ is the natural scale of the penalization constants) is analyzed in \cite{piet:89}.
}
\end{remark}

\begin{remark}
{\rm For $0<x<a$, it can be shown, using arguments similar to those used in \cite{EJZ:98} that under the assumption that $h_0^\prime$ is continuous and strictly positive at $x$,
$$
n^{1/3}\left(\frac{2f_0(x)}{h_0(x)^2h_0^{\prime}(x)}\right)^{1/3}\left(\hat{h}_n(x)-h_0(x)\right)\Dconv 2V,
$$
where $V=\argmax_t\left(W(t)-t^2\right)$, with $W$ standard two-sided Brownian Motion,  has the Chernoff distribution (\cite{chern:64} and \cite{grwe:01}).  This asymptotic distribution is the same as that of the MLE of an increasing hazard function as given in Theorem 6.1 in \cite{prakasarao:70}.
}
\end{remark}

\noindent
Having uniform consistency on arbitrarily large intervals in $[0,a]$ staying away from the boundary and consistency at the boundary points, monotonicity can be used to get uniform consistency of $\hat h_n$ on the whole interval $[0,a]$. We prove a somewhat stronger (not sharp but easy to prove) uniform rate result, that is needed in the proof of Theorem \ref{th:consistency_derivative}.

\begin{corollary}
\label{lem:conshn}
Under the conditions of Lemma \ref{lem:conshnint} and Theorem \ref{th:optimal_penalization},
\begin{equation}
\label{eq:suprate}
\sup_{[0,a]}|\hat{h}_n(x)-h_0(x)|=O_P\left(n^{-1/4}\right).
\end{equation}
\end{corollary}
\begin{proof} For $x\in[n^{-1/4},a-n^{-1/4}]$,
\begin{eqnarray*}
\hat{h}_n(x)&\le&\frac{\hat{H}_n(x+n^{-1/4})-\hat{H}_n(x)}{n^{-1/4}}\\&\le&n^{1/4}\left(\hat{H}_n(x+n^{-1/4})-{H}_0(x+n^{-1/4})-\hat{H}_n(x)+\hat{H}_0(x)\right)+\\
&&\,\,\,\,\,\,\,\,\,\,\,\,\,\,\,\,\,\,\,\,\,\,+n^{1/4}\left(H_0(x+n^{-1/4})-H_0(x)\right)\\&\le& 2n^{1/4}\sup_{[0,a]}|\hat{H}_n(y)-H_0(y)|+h_0(x)+n^{-1/4}\sup_{[0,a]}h_0^\prime(y)\\
&=&h_0(x)+T_n
\end{eqnarray*}
where $T_n=O_P(n^{-1/4})$, not depending on $x$. Similarly, $\hat{h}_n(x)-h_0(x)\ge-T_n$. This leads to the inequality $\sup_{[n^{-1/4},a-n^{-1/4}]}|\hat{h}_n(x)-h_0(x)|\le T_n$. For $x\in[0,n^{-1/4}]$, we have
\begin{eqnarray*}
\hat{h}_n(x)&\le&\hat{h}_n(n^{-1/4})-h_0(n^{-1/4})+h_0(n^{-1/4})-h_0(x)+h_0(x)\\&\le& h_0(x)+T_n+n^{-1/4}\sup_{[0,a]}h_0^\prime(y)
\end{eqnarray*}
and, using Theorem \ref{th:optimal_penalization} part (iii),
\begin{eqnarray*}
\hat{h}_n(x)&\ge&\hat{h}_n(0)-h_0(0)+h_0(0)-h_0(x)+h_0(x)\\&=&h_0(x)+O_P(n^{-1/3})-n^{-1/4}\sup_{[0,a]}h_0^\prime(y)
\end{eqnarray*}
leading to $\sup_{[0,n^{-1/4}]}|\hat{h}_n(x)-h_0(x)|=O_P(n^{-1/4})$. For $[a-n^{-1/4},a]$ the result can be derived in the same way. Combining the three rate results for the suprema leads to (\ref{eq:suprate}).
\end{proof}

\begin{figure}
\begin{center}
\includegraphics[width=6cm]{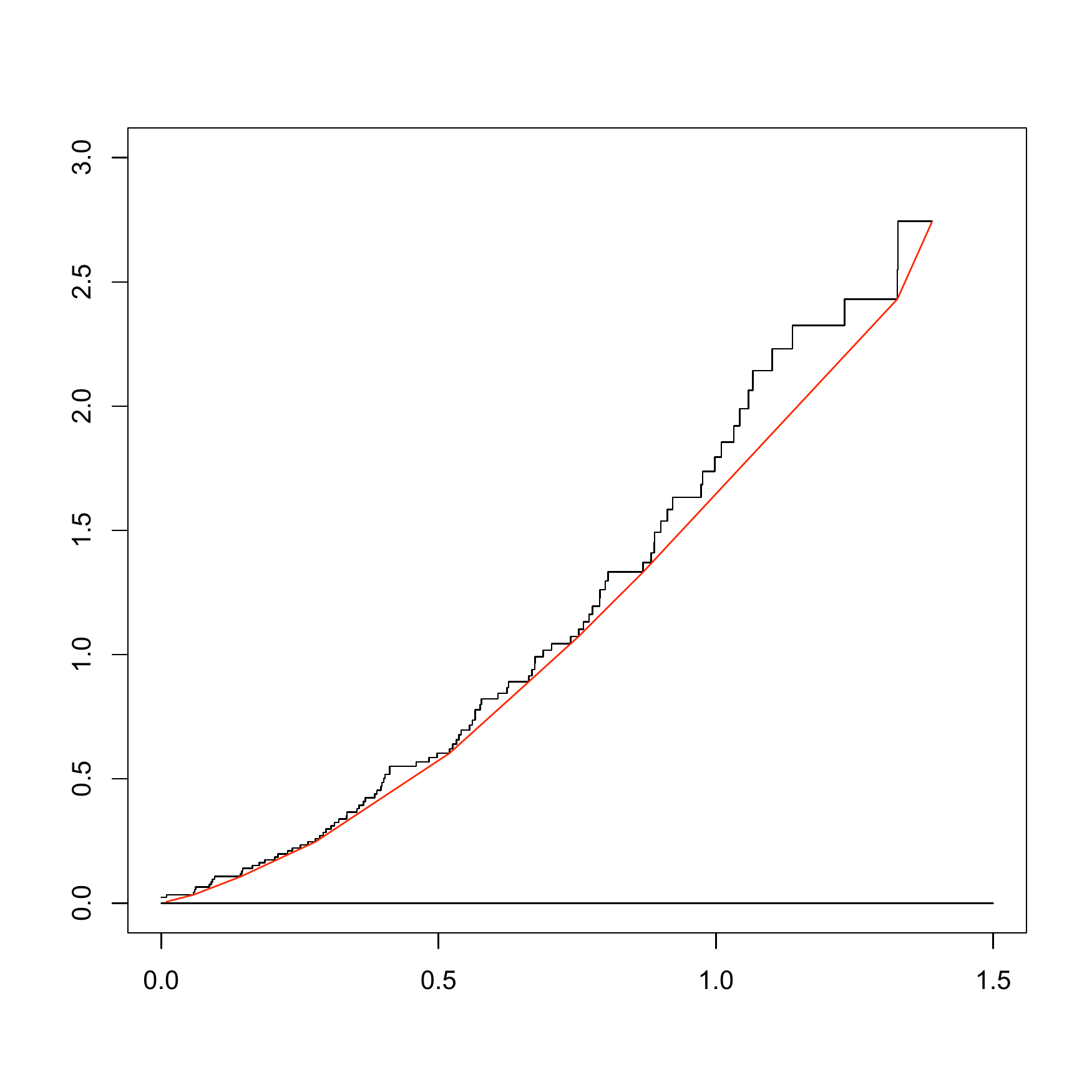}
\includegraphics[width=6cm]{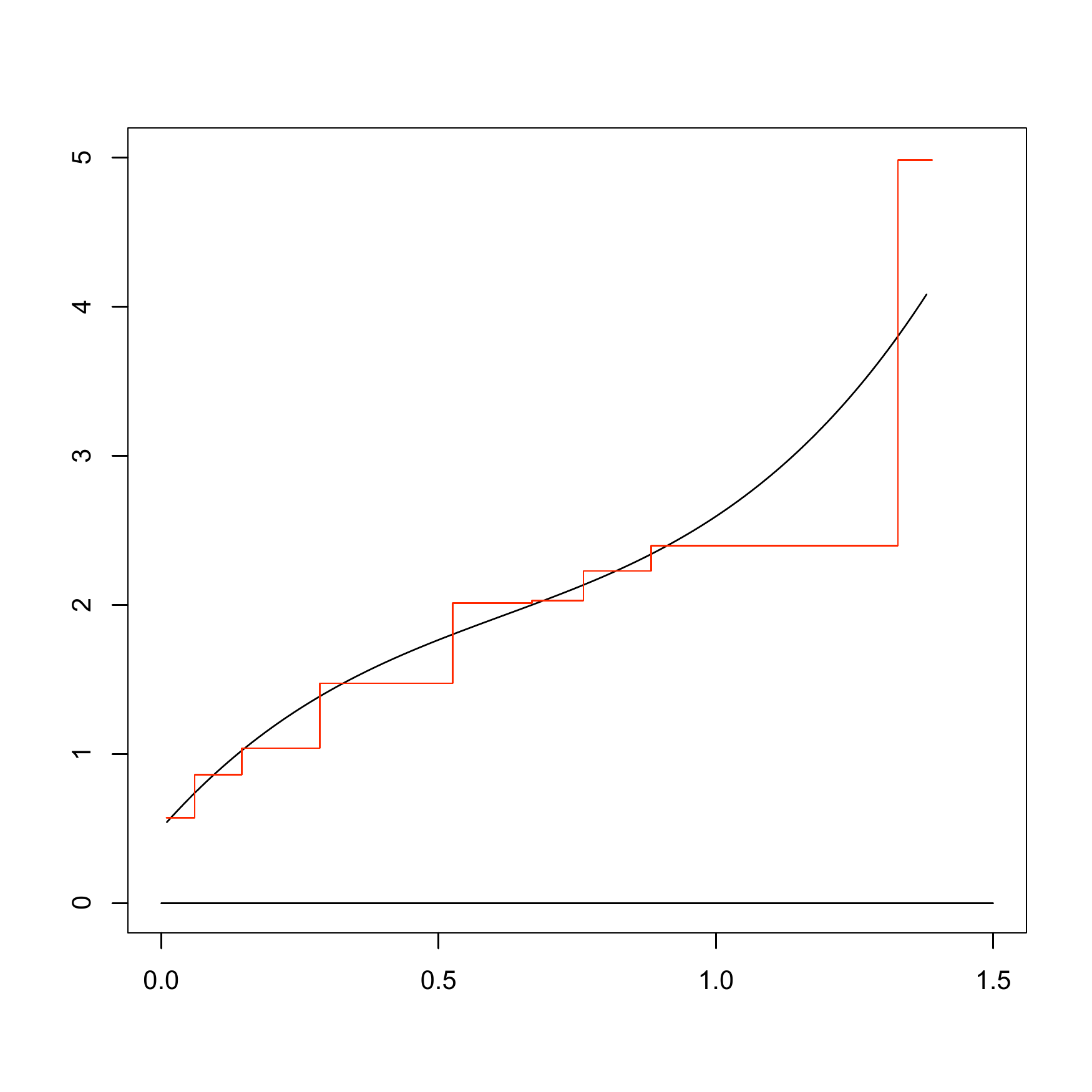}
\end{center}
\caption{The penalized cusum diagram and its greatest convex minorant (left panel) and the penalized least squares estimate of the hazard (right panel)  for a sample of size $n=100$ from the distribution function $F^{(1)}$ on the interval $[0,\left(F^{(1)}\right)^{-1}(0.95)]$. The real hazard is the black curve in the right panel.}
\label{fig:hazard_pen100}
\end{figure}

Taking $\a_n=n^{-2/3}$ and $\b_n=2n^{-2/3}$ we obtain Figure  \ref{fig:hazard_pen100}, for the same sample as used in Figure \ref{fig:hazard100_unpen}, where one notices that the value of $h(0)$ has gone up and the value of $h(a)$ has gone down.

\section{Monotone kernel estimates of the hazard}
\label{section:kernel_estimates}
Now suppose we have an initial (non-smooth) monotone estimate of the hazard $\hat h_n$ on $[0,a]$, like the least squares estimate of the hazard, obtained by minimizing (\ref{quad_criterion1}), or the penalized least squares isotonic estimator, obtained by minimizing (\ref{LS_criterion2}) under the restiction that $h$ is nondecreasing.
One way of constructing a smooth estimate of the hazard based on $\hat h_n$ is to use kernel smoothing. A kernel estimate with bandwidth $b>0$ of the hazard is given by:
\begin{equation}
\label{kernel_est_hazard}
\tilde h_n(x)=\int K_b(x-y)\,d\hat H_n(y)=\int K_b(x-y)\,\hat h_n(y)\,dy,\quad K_b(u)=b^{-1}K(u/b),
\end{equation}
where $K$ is a kernel with compact support, like the triweight kernel
$$
K(u)=\frac{35}{32}\left\{1-u^2\right\}^31_{[-1,1]}(u).
$$
Note that monotonicity of $\tilde{h}_n$  follows from monotonicity of $\hat h_n$. This property is not shared by the direct kernel estimator for $h_0$ that is obtained by taking the empirical cumulative hazard function $\H_n$ instead of $\hat H_n$ in (\ref{kernel_est_hazard}). Also the kernel estimators  considered in \cite{singpurwallawong:83}, which are  ratios of kernel estimators of the density $f_0$ and estimators of the survival function $1-F_0$, are not monotone in general. An alternative representation of our kernel estimate is
\begin{align*}
\tilde h_n(x)&=\int K_b(x-y)\,\int_{-\infty}^y d\hat h_n(u)\,dy
=\iint_{u<y} K_b(x-y)\,dy\,d\hat h_n(u)\\
&=\int_{u=-\infty}^{x+b}\int_{y=u}^{x+b}K_b(x-y)\,dy\,d\hat h_n(u)
=\int_{u=-\infty}^{x+b}\IK((x-u)/b)\,d\hat h_n(u)
\end{align*}
for $x\in[0,a]$, where
$$
\IK(u)=\int_{-\infty}^u K(w)\,dw=\left\{\begin{array}{lll}
0\,&,\,u< -1\\
\displaystyle{\int_{-1}^u K(w)\,dw}&,\,u\in[-1,1],\\
1\,&,\,u>1.
\end{array}
\right.
$$
Note that this yields:
$$
\tilde h_n'(x)=\widetilde {h_n'}(x)=\int K_b(x-y)\,d\hat h_n(y),
$$
so we also have an estimate of the derivative of the hazard.

Going in the other direction, we
have the following estimate of the cumulative hazard function:
$$
\tilde H_n(x)=b\int_{z\in[0,x+b]} \JK((x-z)/b)\,d\hat h_n(z),
$$
where
$$
\JK(u)=\int_{-\infty}^u\IK(v)\,dv=\left\{\begin{array}{lll}
0,\,&u< -1\\
\displaystyle{\int_{-1}^u (u-v) K(v)\,dv},\,&u\in[-1,1],\\
u,\,&u>1
\end{array}
\right.
$$
We now have:
\begin{align*}
&\tilde H_n(x)=\hat H_n(x-b)+b\int_{x-b}^{x+b} \JK((x-z)/b)\,d\hat h_n(z)\\
&=\hat H_n(x-b)+\int_{x-b}^{x+b} \hat h_n(z)\int_{-1}^{(x-z)/b} K(v)\,dv\,dz
=\int_{x-b}^{x+b} \hat H_n(z) K_b(x-z)\,dz.
\end{align*}

As the estimate for the density on $[0,a]$, we can take:
\begin{equation}
\label{density-est}
\tilde f_n(x)=\tilde h_n(x)\exp\left\{-\tilde H_n(x)\right\}.
\end{equation}

\begin{figure}[!ht]
\begin{center}
\includegraphics[width=4cm]{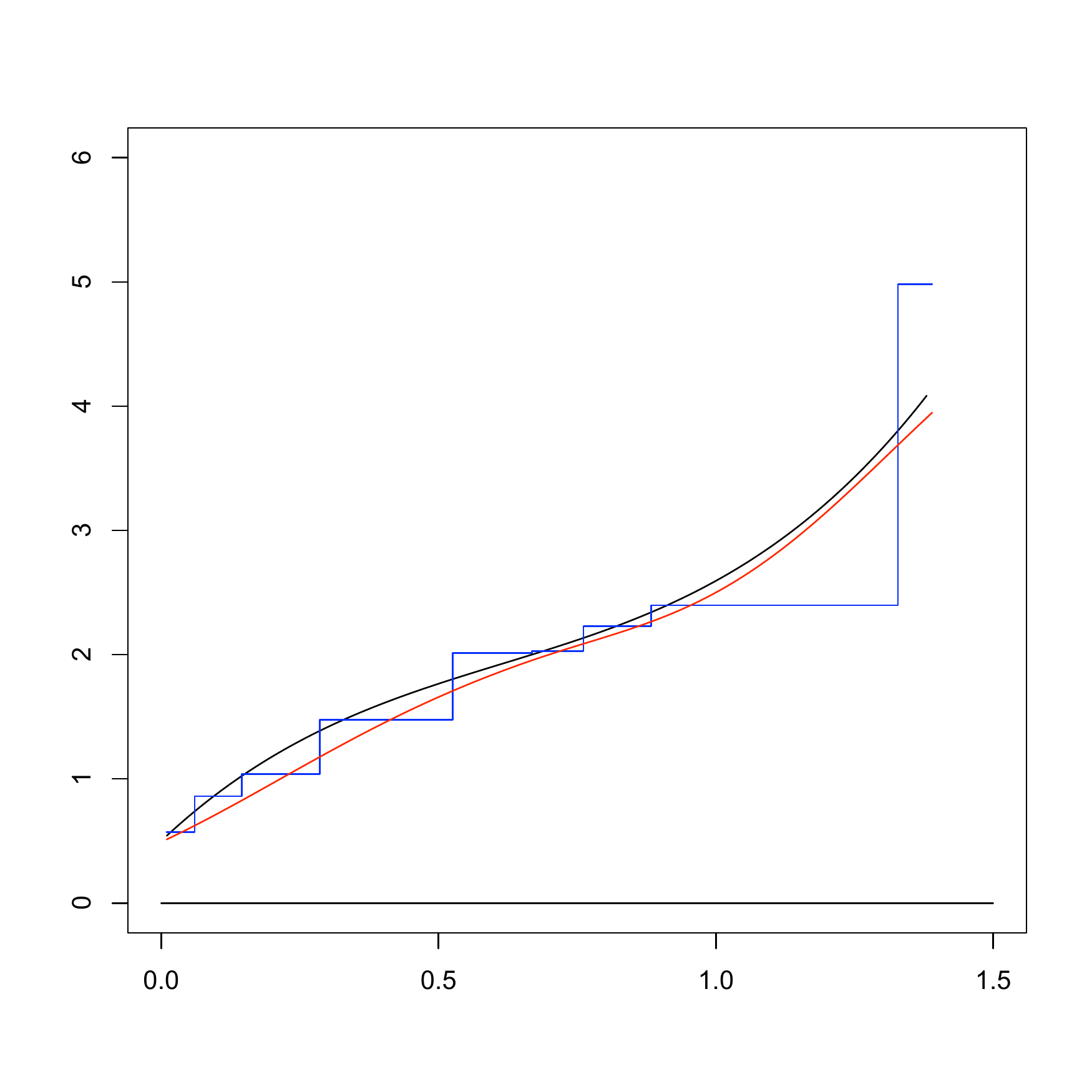}
\includegraphics[width=4cm]{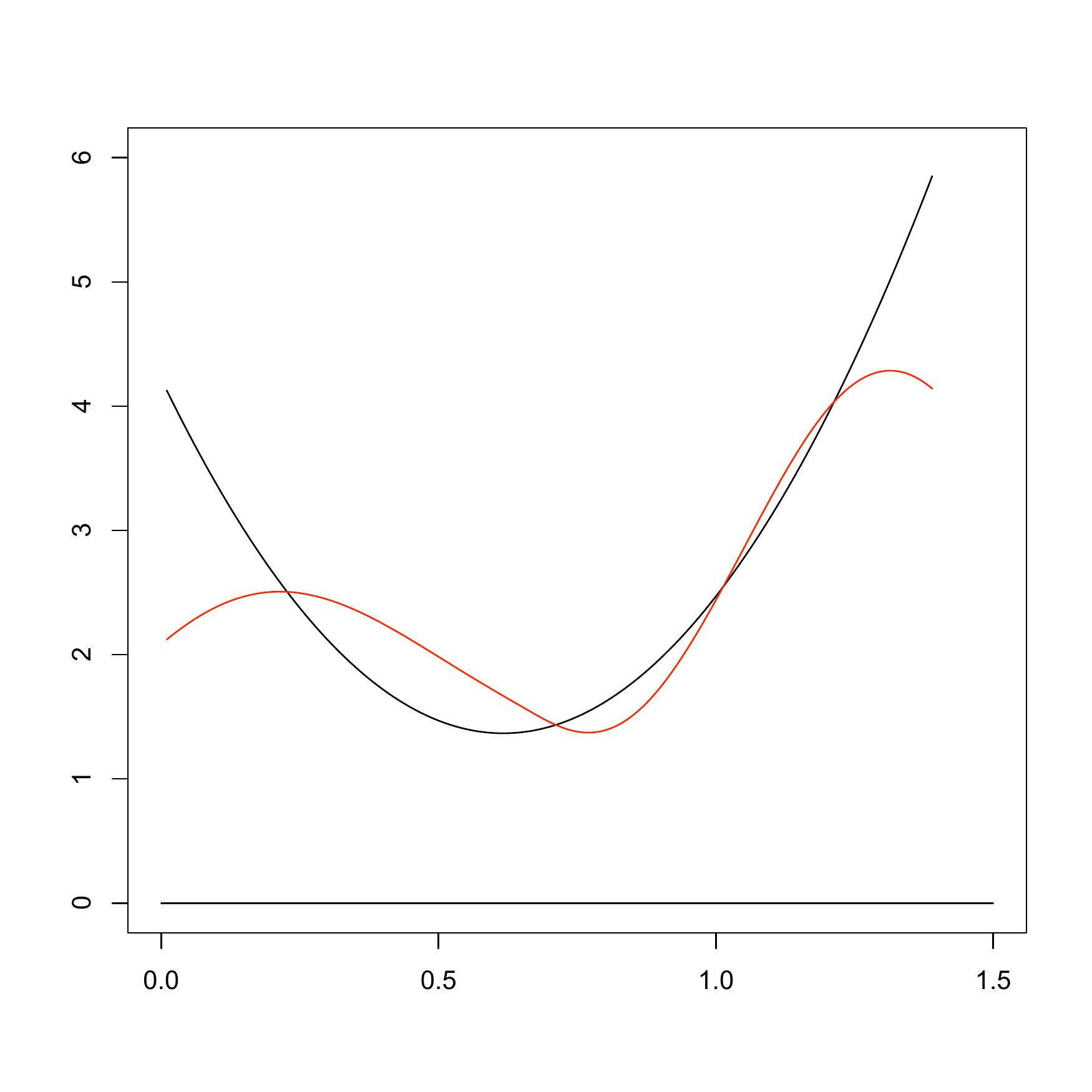}
\includegraphics[width=4cm]{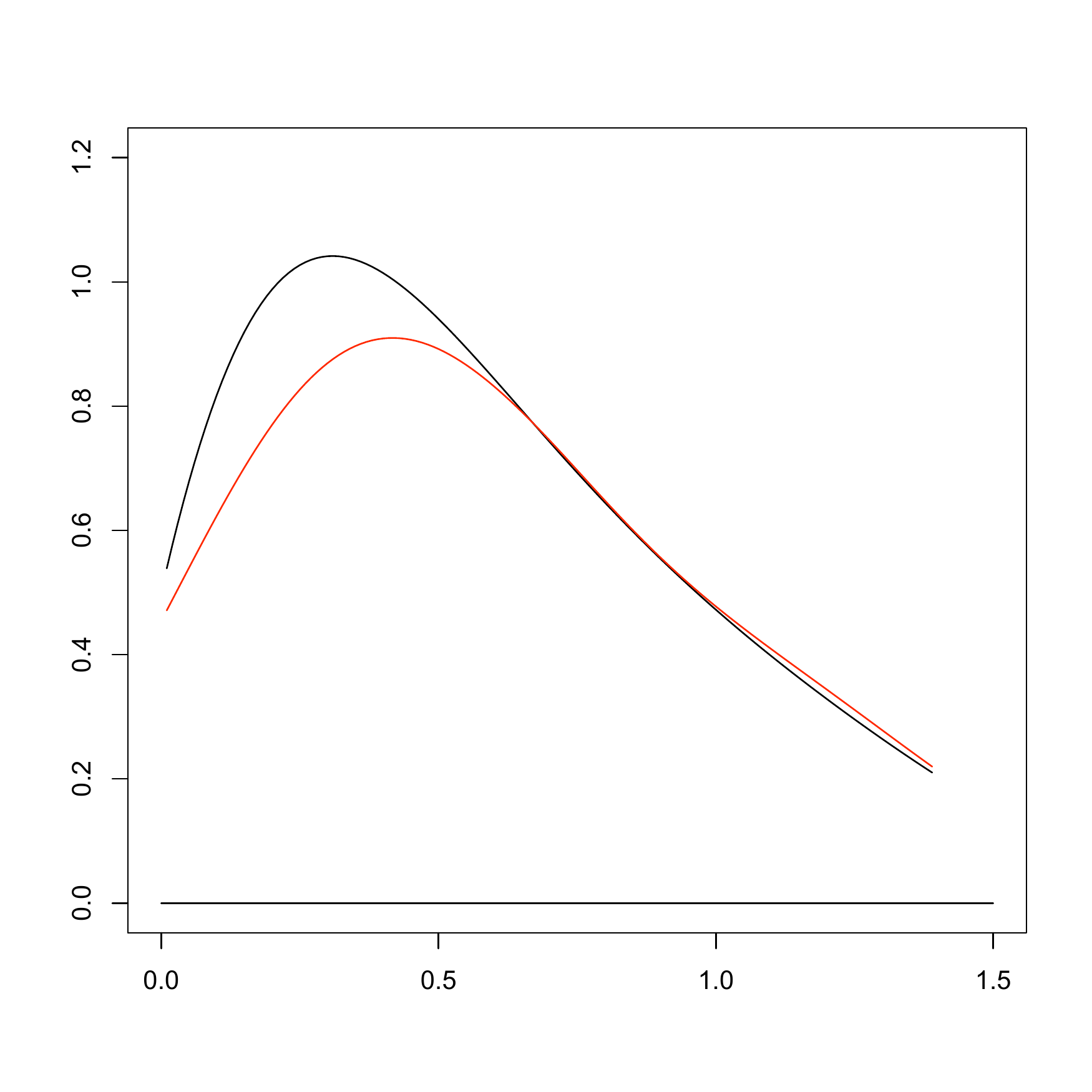}
\end{center}
\caption{The left picture shows the estimates $\hat h_n$ (blue) and $\tilde h_n$ (red) of the hazard $h^{(1)}$ (of the family $\{h^{(d)}:d\in[-1,1]\}$, black) for a sample of size $n=100$ on the $95\%$ percentile interval $[0,\left(F^{(1)}\right)^{-1}(0.95)]$. The middle and right picture show the corresponding derivatives of the hazard rates and  the corresponding densities.}
\label{fig:haz1000}
\end{figure}

\begin{figure}[!ht]
\begin{center}
\includegraphics[width=4cm]{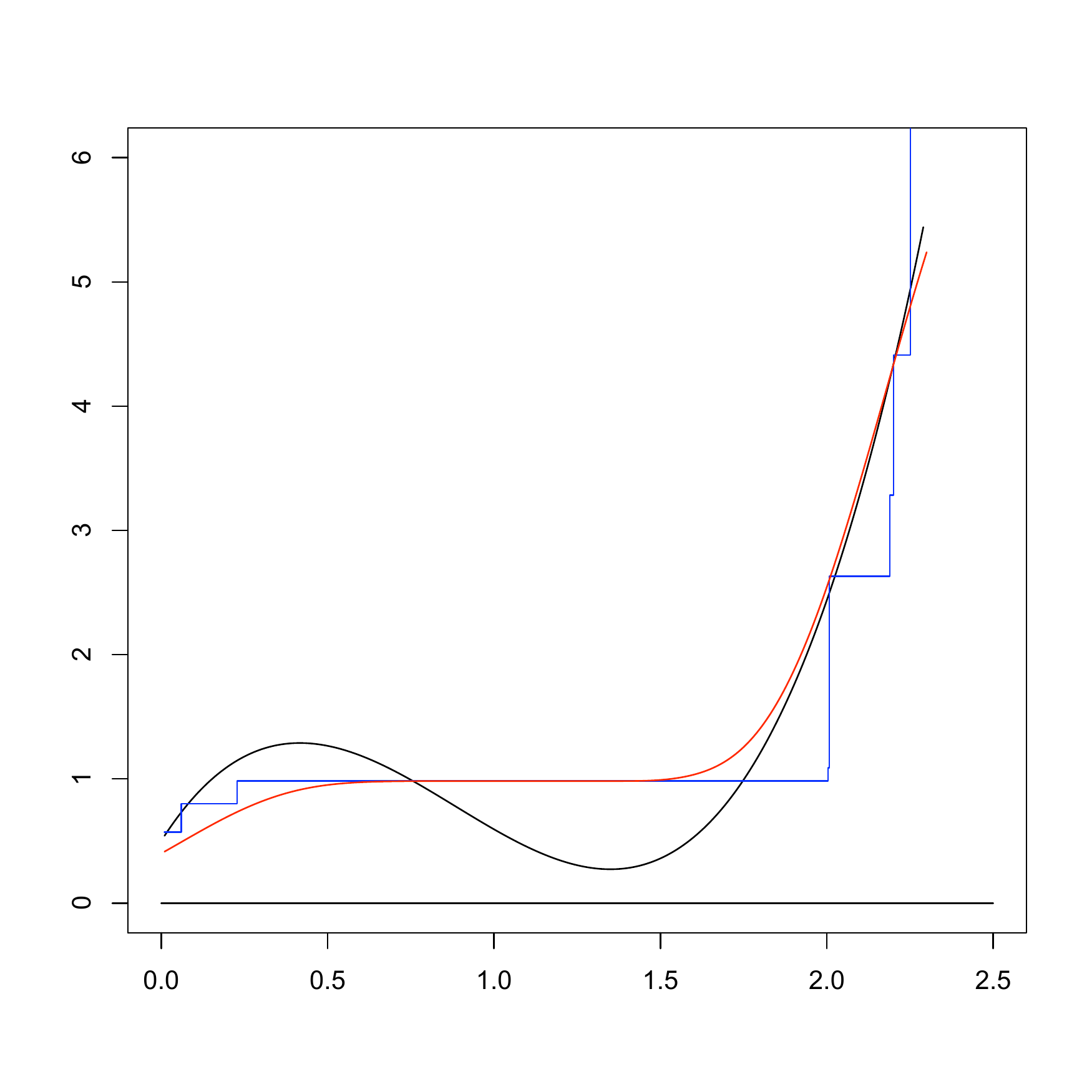}
\includegraphics[width=4cm]{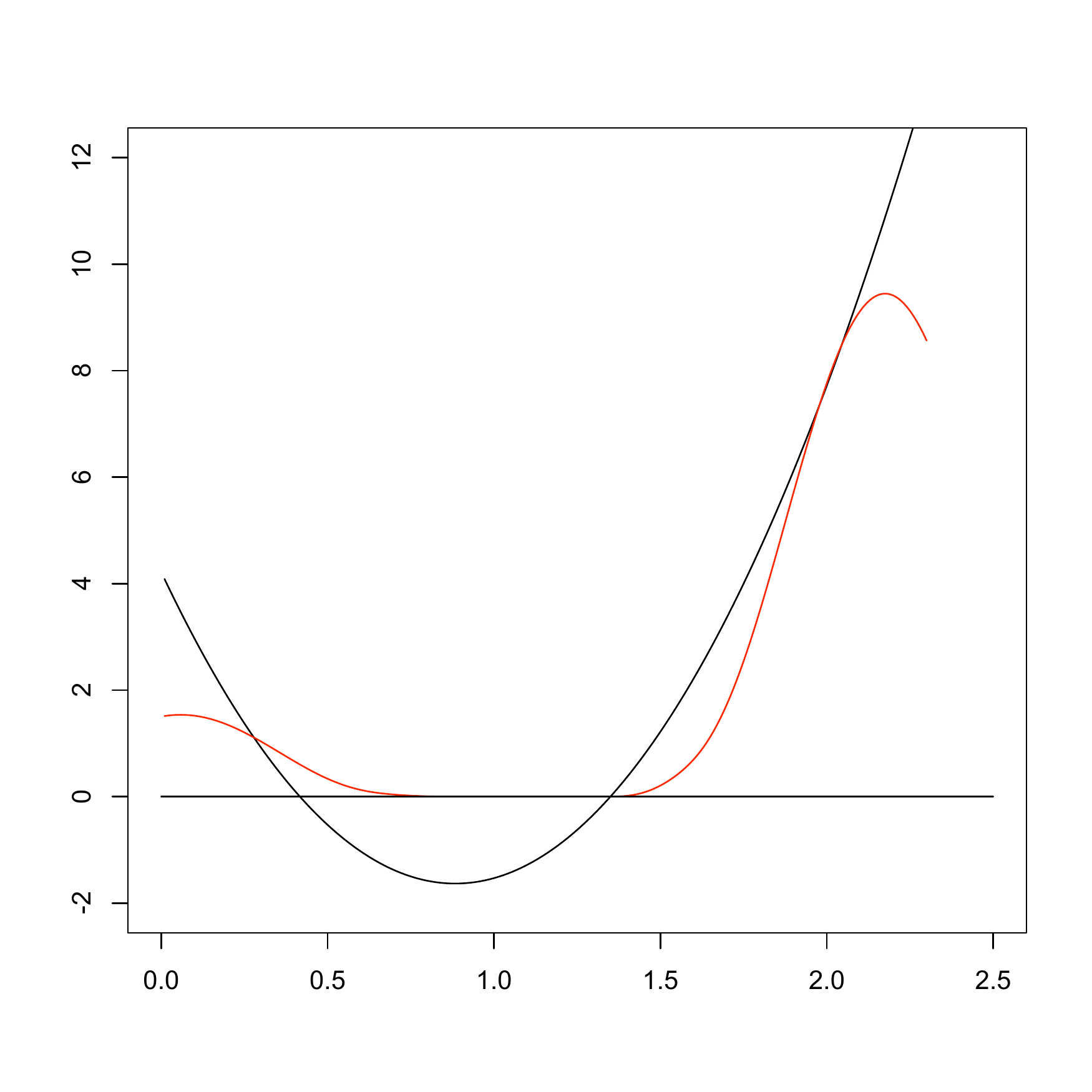}
\includegraphics[width=4cm]{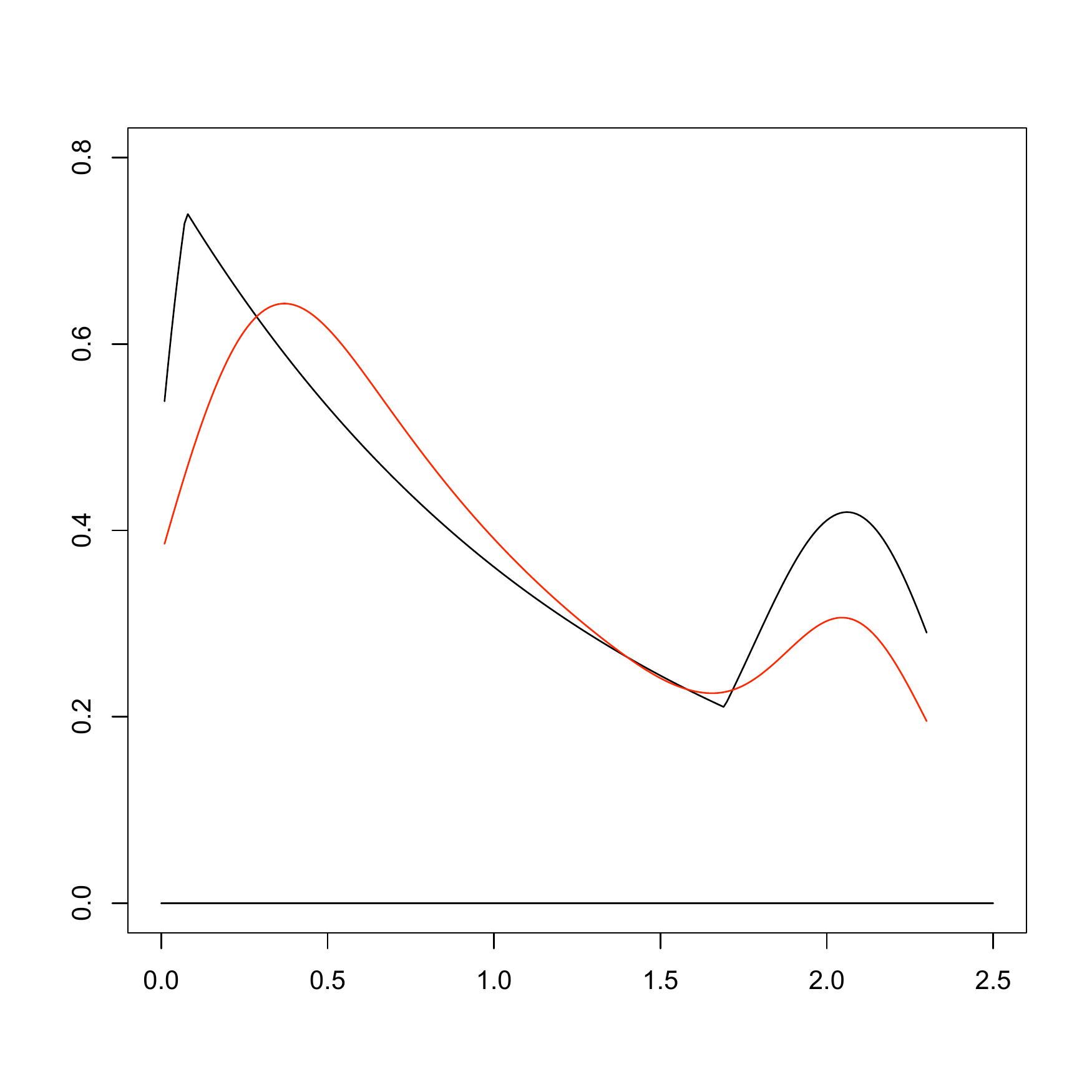}
\end{center}
\caption{From left to right the  isotonic estimates $\hat h_n$ (blue) and $\tilde h_n$ (red) of the hazard $h^{(-1)}$ (of the family $\{h^{(d)}:d\in[-1,1]\}$, black) for a sample of size $n=100$, and the real hazard $h^{(-1)}$ (black) on the $95\%$ percentile interval $[0,\left(F^{(-1)}\right)^{-1}(0.95)]$; the corresponding derivatives of the hazard rates and  the corresponding densities, where we compare in the right panel the estimate of the density with the density, obtained from the isotonic projection of the underlying hazard $h^{(1)}$. }
\label{fig:haz-1}
\end{figure}

For $x\in(0,a)$, we have the following asymptotic result for $\tilde{h}_n(x)$.
\begin{thm}
\label{th:kernel_est}
Let $\tilde h_n$ be the kernel estimate of the hazard function on $[0,a]$, defined by (\ref{kernel_est_hazard}).
Moreover, let $h_0$ be twice continuously differentiable, and let $h_0$ and $h_0'$ be both strictly positive on $[0,a]$, where $h_0'(0)$ and $h_0'(a)$ are defined as right and left derivatives, respectively. Then:
\begin{enumerate}
\item[(i)]
If we choose a bandwidth $b_n$ such that $n^{1/5} b_n\to \nu\in(0,\infty)$, as $n\to\infty$, we have for each $x\in(0,a)$:
\begin{align*}
n^{2/5}\left\{\tilde h_n(x)-h_0(x)\right\}
\stackrel{\cal D}\longrightarrow N\left(\m_0(\nu),\s_0^2(\nu)\right),
\end{align*}
where
\begin{equation}\label{eq:asmusig}
\m_0(\nu)=\tfrac12\nu^2h_0''(x)\int u^2 K(u)\,du,\qquad \s_0^2(\nu)=\frac{h_0(x)^2}{\nu f_0(x)}\int K(u)^2\,du.
\end{equation}
\item[(ii)]
The asymptotically locally optimal bandwidth is given by
\begin{equation}
\label{b_loc_opt}
b_{n,locopt}=b_{n,locopt}(x)=\left\{\frac{ h_0(x)^2\int K(u)^2\,du}{f_0(x) h_0''(x)^2\left\{\int u^2 K(u)\,du\right\}^2}\right\}^{1/5}n^{-1/5}.
\end{equation}
The bandwidth, minimizing the asymptotic global least squares criterion
$$
\frac1{\nu}\int_0^a\frac{h_0(x)^2}{f_0(x)}\,dx\int K(u)^2\,du+\tfrac14\nu^4\left\{\int u^2K(u)\,du\right\}^2\int_0^a h_0''(x)^2\,dx,
$$
is given by
\begin{equation}
\label{b_globopt}
\nu_{n,globopt}\cdot n^{-1/5}=\left\{\frac{\int_0^a h_0(x)^2/f_0(x)\,dx\int K(u)^2\,du}{\int_0^a h_0''(x)^2\,dx\left\{\int u^2 K(u)\,du\right\}^2}\right\}^{1/5}n^{-1/5}.
\end{equation}
\end{enumerate}
\end{thm}

\noindent
\begin{proof}
(i): We get:
\begin{align*}
&\int K_{b_n}(x-y)\,d\hat H_n(y)
=\int K_{b_n}(x-y)\,d\H_n(y)+\int K_{b_n}(x-y)\,d\left(\hat H_n-\H_n\right)(y)\\
&=\int K_{b_n}(x-y)\,d\H_n(y)+\int\left\{\hat H_n(y)-\H_n(y)\right\}\,\frac1{b_n^2}K'((x-y)/b_n)\,dy\\
&=\int K_{b_n}(x-y)\,d\H_n(y)+\frac1{b_n}\int\left\{\hat H_n(x-b_nu)-\H_n(x-b_nu)\right\}\,K'(u)\,dy\\
&=\int K_{b_n}(x-y)\,d\H_n(y)+O_P\left(n^{-7/15}\log n\right),
\end{align*}
where we use that
$$
\sup_{x\in[0,a]}\left|\hat H_n(x)-\H_n(x)\right|=O_P\left(n^{-2/3}\log n\right).
$$
This result is related to that in \cite{kiefwolf:76} for the concave majorant of the empirical distribution based on a sample from a concave distribution function. It can be proved along the lines of  \cite{palwood:06}.
Moreover,
\begin{align*}
&\int K_{b_n}(x-y)\,d\H_n(y)=\int K_{b_n}(x-y)\,dH_0(y)+\int K_{b_n}(x-y)\,d\left(\H_n-H_0\right)(y).
\end{align*}
Define
\begin{eqnarray*}
\tilde{W}_n(u)&=&\sqrt{n/b_n}\left(\H_n(x+b_nu)-\H_n(x)-H_0(x+b_nu)+H_0(x)\right)\\
&=&\sqrt{n/b_n}\left\{-\log\left(\frac{1-\F_n(x+b_nu)}{1-\F_n(x)}\right)+\log\left(\frac{1-F_0(x+b_nu)}{1-F_0(x)}\right)\right\}\\
&=&\sqrt{n/b_n}\left\{-\log\left(1-\frac{\F_n(x+b_nu)-\F_n(x)}{1-\F_n(x)}\right)+\right.\\&&\,\,\,\,\,\,\,\,\left.+\log\left(1-\frac{F_0(x+b_nu)-F_0(x)}{1-F_0(x)}\right)\right\}\\
&=&\sqrt{n/b_n}\left\{-\log\left(1-\frac{T_n(u)}{1-\F_n(x)}\right)+\log\left(1-\frac{t_n(u)}{1-F_0(x)}\right)\right\}\\
&=&\sqrt{n/b_n}\left\{\frac{T_n(u)}{1-\F_n(x)}-\frac{t_n(u)}{1-F_0(x)}\right\}+o_P(1)=\sqrt{n/b_n}\frac{T_n(u)-t_n(u)}{1-F_0(x)}+\\ &&\,\,\,\,+\sqrt{n/b_n}\left(T_n(u)-t_n(u)\right)\frac{\F_n(x)-F_0(x)}{(1-F_0(x))(1-\F_n(x))}+\\
&&\,\,\,\,+\sqrt{n/b_n}t_n(u)\frac{\F_n(x)-F_0(x)}{(1-F_0(x))(1-\F_n(x))}+o_P(1)\\
&=&\sqrt{n/b_n}\frac{T_n(u)-t_n(u)}{1-F_0(x)}+o_P(1)
\end{eqnarray*}
where the order terms are uniform for $u$ in compact sets. Using that
\begin{eqnarray*}
\sqrt{n/b_n}\left(T_n(u)-t_n(u)\right)&=&\int b_n^{-1/2}\left(1_{[0,x+b_nu]}(y)-1_{[0,x]}(y)\right)\,d\sqrt{n}(\F_n-F_0)(y)\\
&\Dconv&\sqrt{f_0(x)}W(u),
\end{eqnarray*}
where $W$ is standard two-sided Brownian motion on $\R$, we obtain
$$
\tilde{W}_n(u)\stackrel{\cal D}\longrightarrow\frac{\sqrt{f_0(x)}}{1-F_0(x)}W(u)\stackrel{\cal D}=\frac{h_0(x)}{\sqrt{f_0(x)}}W(u).
$$
Take $b_n=\nu n^{-1/5}$ and note that
\begin{eqnarray*}
&&n^{2/5}\int K_{b_n}(x-y)\,d\left(\H_n-H_0\right)(y)=n^{2/5}b_n^{-1}\int K\left(\frac{x-y}{b_n}\right)\,d\left(\H_n-H_0\right)(y)\\
&&\,\,\,\,=n^{2/5}b_n^{-1}\int K\left(u\right)\,d\left(\H_n-H_0\right)(x+b_nu)=\nu^{-1/2}\int K\left(u\right)\,d\tilde{W}_n(u)\\
&&\,\,\,\,\stackrel{\cal D}\longrightarrow\frac{h_0(x)}{\sqrt{\nu f_0(x)}}\int K(u)\,dW(u)\stackrel{\cal D}=N\left(0,\frac{h_0(x)^2}{\nu f_0(x)}\int K(u)^2\,du\right).
\end{eqnarray*}
The asymptotic bias is given by
\begin{eqnarray*}
n^{2/5}\int K_b(x-y) h_0(y)\,dy-h_0(x)
&=&n^{2/5}\int K(u)\left\{h_0(x+b_nu)-h_0(x)\right\}\,du\\
&\sim&\tfrac12h_0''(x) \nu^2\int u^2K(u)\,du.
\end{eqnarray*}
So we obtain
\begin{align*}
n^{2/5}\left\{\tilde h_n(x)-h_0(x)\right\}
\stackrel{\cal D}\longrightarrow N\left(\m_0(\nu),\s_0^2(\nu)\right),
\end{align*}
where $\m_0(\nu)$ and $\s_0^2(\nu)$ are given in (\ref{eq:asmusig}).
The last two statements of the theorem follow easily by setting the derivative with respect to  $\nu$ equal to zero in, respectively, the local and global criterion.\end{proof}

\vspace{0.3cm}
Pictures for $n=100$ of  $\tilde h_n$, its derivative $\tilde h_n'$ and the density $\tilde f_n$ for the corresponding functions at the right end of the family (\ref{F_d})
are shown in Figure \ref{fig:haz1000}, where the globally optimal bandwidth for the hazard, given in (\ref{b_globopt}) is used. The same pictures for the left end of the family (where the hazard is not monotone), are shown in Figure \ref{fig:haz-1}.

For purposes of bootstrapping of the test statistics in \cite{GrJo10c}, a crucial feature is that the estimate of the derivative of the hazard stays away from zero, also at the boundary points. This behavior  can be shown under the hypotheses of Theorem \ref{th:kernel_est}, even at the boundary points. To obtain a consistent estimate of the derivative at the boundary points, one could  introduce a boundary kernel. For example, near the left boundary point one could take:
\begin{align*}
&\widetilde{h_n'}(x)=\a(x/b)\int K_b(x-y)\,d\hat h_n(y)+\b(x/b)\int \frac{x-y}{b}K_b(x-y)\,d\hat h_n(y)\\ & \mbox{ where }K_b(u)=b^{-1}K(u/b),
\end{align*}
and  $\a(x/b)$ and $\b(x/b)$ are chosen in such a way that, if $y\in[0,1]$,
\begin{align*}
&\a(y)\int_{-1}^y K(u)\,du+\b(y)\int_{-1}^y uK(u)\,du=1 \mbox{ and }\\
&\a(y)\int_{-1}^y uK(u)\,du+\b(y)\int_{-1}^y u^2K(u)\,du=0,
\end{align*}
and where $\a(y)=1,\,\b(y)=0$, if $y>1$.
This will indeed lead to consistent estimates of the derivative of the hazard at the boundary, but the disadvantage is that the relation between $\tilde h_n$ and its derivative via derivatives and integrals of the kernel, which we used above, is lost.
In generating the bootstrap samples in \cite{GrJo10c}, using a kernel estimate of the hazard, boundary kernels were not used for estimating the derivative at the boundary, since it did not lead to significantly different results, and destroyed the simple relation between the hazard and its derivative via the kernel.

\section{Smooth estimates of the hazard, based on penalization}
\label{section:penalization_estimates}
Another approach to obtain a smooth monotone estimate of the hazard is that of penalized least squares; see e.g.\
 \cite{tantwood:94} and \cite{palwood:07}. Let $\lambda\ge0$ be a penalty parameter and define the smooth penalized local least squares estimator of $h$ on $[0,a]$ as minimizer of
\begin{equation}
\label{penalized_criterion}
\Phi_{\lambda}(h)=\int_{0}^a\left(h(x)-\hat h_n(x)\right)^2\,dx+\lambda\int_0^a h^{\prime}(x)^2\,dx
\end{equation}
over the set of differentiable functions $h$ on $[0,a]$, where $\hat h_n$ is the monotone (on [0,a]) piecewise constant estimate that minimizes (\ref{LS_criterion2}) in section \ref{section:isotonic_estimates}. Our first lemma gives the minimizer of $\Phi_{\lambda}$ over the class of smooth functions on $[0,a]$ under boundary constraints at  $0$ and $a$.

\begin{lemma}
\label{lem:charvar}
Let $\kappa_1,\kappa_2\in\R$. Then the unique minimizer of $\Phi_{\lambda}$ over all smooth functions on $[0,a]$ such that $h(0)=\kappa_1$ and $h(a)=\kappa_2$ exists and is given by
\begin{equation}
\label{eq:partsol}
h(x)=h_1(x)+c_1e^{-x/\sqrt{\lambda}}+c_2e^{-(a-x)/\sqrt{\lambda}}
\end{equation}
where
\begin{equation}
\label{eq:h1}
h_1(x)=\tfrac12\lambda^{-1/2}\int_0^a e^{-|y-x|/\sqrt{\lambda}}\hat h_n(y)\,dy,
\end{equation}
and $c_1$ and $c_2$ are chosen such that $h$ satisfies the imposed boundary constraints.
\end{lemma}
\begin{proof}
Writing
$$
I(h)=\int_0^a G(x,h,h')\,dx=\int_0^a \bigl\{h(x)-\hat h_n(x)\bigr\}^2\,dx+\l \int_0^a h'(x)^2\,dx,
$$
we get Euler's differential equation
$$
G_h-\frac{d}{dx} G_{h'}=0,
$$
we wish to solve under under the boundary conditions $h(0)=\kappa_1$ and $h(a)=\kappa_2$. This results in the second order integral equation
\begin{equation}
\label{2nd_order_eq}
h''(x)=\l^{-1}\bigl\{h(x)-\hat h_n(x)\bigr\}
\end{equation}
with boundary constraints.

A particular solution to (\ref{2nd_order_eq})  is given by (\ref{eq:h1}).
Adding the solutions to the homogeneous equation multiplied by constants $c_1$ and $c_2$ respectively, the unique solution to the boundary value problem is obtained by choosing $c_1$ and $c_2$ appropriately in (\ref{eq:partsol}).
\end{proof}

\begin{remark}\rm
Observe that $h_1$ in  (\ref{eq:h1}) can be viewed as a kernel-smoothed version of $\hat h_n$ in the sense of section \ref{section:kernel_estimates}, with kernel function $K(x)=\tfrac12\exp(-|x|)$ and bandwidth $b=\sqrt{\lambda}$. In particular this shows $h_1$ to be monotone. Moreover, for $\lambda\downarrow0$ and $c_1,c_2$ bounded as $\lambda\downarrow0$,  $h$ defined in (\ref{eq:partsol}) is merely a boundary-corrected version of $h_1$. In that case the asymptotic behavior of $h$ on closed intervals excluding the boundary points $0$ and $a$ is completely determined by that of $h_1$.
\end{remark}

As an immediate consequence of Lemma \ref{lem:charvar}, the minimizer of $\Phi_{\l}$ without boundary restrictions can be identified, as well as the minimizer under the natural boundary constraints  $h(0)=\hat{h}_n(0)$ and $h(a)=\hat{h}_n(a)$. This latter boundary constraints are natural in view of the consistency of $\hat{h}_n$ at $0$ and $a$.

\begin{corollary}
\label{lem:variation}
The unique minimizer $\check{h}_n$ of $\Phi_{\lambda}$ over all smooth functions on $[0,a]$ exists is given by (\ref{eq:partsol})  with $c_1$ equal to
\begin{equation}
\label{eq:c1}
\check{c}_1=\frac{\int_0^a \left(\hat h_n(x)-h_1(x)+\sqrt{\l}h_1^\prime(x)\right) e^{-x/\sqrt{\l}}\,dx}{\sqrt{\l}\left\{1-e^{-2a/\sqrt{\l}}\right\}}.
\end{equation}
and $c_2$ equal to
\begin{equation}
\label{eq:c2}
\check{c}_2=\frac{\int_0^a\left(\hat h_n(x)-h_1(x)-\sqrt{\l}h_1^\prime(x)\right)e^{-(a-x)/\sqrt{\l}}\,dx}{\sqrt{\l}\left\{1-e^{-2a/\sqrt{\l}}\right\}}.
\end{equation}
The minimizer $\bar{h}_n$ of $\Phi_{\lambda}$ under the boundary constraints $h(0)=\hat{h}_n(0)$ and $h(a)=\hat{h}_n(a)$ is given by (\ref{eq:partsol}) with
\begin{equation}
\label{eq:c1_2}
c_1=\bar{c}_1
=\frac{\hat h_n(0)-h_1(0)-\bigl\{\hat h_n(a)-h_1(a)\bigr\}e^{- a/\sqrt{\lambda}}}{1-e^{-2a/\sqrt{\lambda}}}\,,
\end{equation}
and
\begin{equation}
\label{eq:c2_2}
c_2=\bar{c}_2
=\frac{\hat h_n(a)-h_1(a)-\bigl\{\hat h_n(0)-h_1(0)\bigr\}e^{-a/\sqrt{\lambda}}}{1-e^{-2a/\sqrt{\lambda}}}\,.
\end{equation}
\end{corollary}
\begin{proof}
The parameters $\check{c}_1$ and $\check{c}_2$ are found by differentiating the criterion $\Phi_{\l}$ evaluated at (\ref{eq:partsol}) with respect to $c_1$ and $c_2$. Differentiation w.r.t.\ $c_1$ yields
\begin{align*}
&\int_{0}^a\left\{h(x)-\hat h_n(x)-\sqrt{\l} h'(x)\right\}e^{-x/\sqrt{\l}}\,dx=0
\end{align*}
and differentiation w.r.t.\ $c_2$ yields
\begin{align*}
\int_{0}^a\left\{h(x)-\hat h_n(x)+\sqrt{\l} h'(x)\right\}e^{-(a-x)/\sqrt{\l}}\,dx=0,
\end{align*}
where the dependence on $c_1$ and $c_2$ in the equations is implicit via $h$ and $h^{\prime}$.
From this (\ref{eq:c1}) and (\ref{eq:c2}) follow. To get (\ref{eq:c1_2}) and (\ref{eq:c2_2}), $c_1$ and $c_2$ are chosen in (\ref{eq:partsol}) to satisfy the imposed boundary constraints.  \end{proof}

\noindent
The major part of the asymptotic behavior of the smoothness-penalized estimators are related to the asymptotics of $h_1$. The lemma below establishes uniform consistency of $h_1$.

\begin{lemma}
\label{lem:unifsmallerset}
Let $\hat h_n$ be the (possibly boundary-penalized) least squares estimator of section \ref{section:isotonic_estimates}, where $\a_n,\b_n\downarrow0$. Let $h_1$ be defined by (\ref{eq:h1}). Then, for $\lambda=\lambda_n\downarrow0$ and $(\log n)^2\lambda\rightarrow 0$, we have for all $0<\delta<a/2$
$$
\sup_{[\delta,a-\delta]}|h_1(x)-h_0(x)|=O_P(n^{-1/4})
$$
If, moreover, $\a_n$ and $\b_n$ satisfy the conditions of Corollary \ref{lem:conshn},  then  for $x=0$ and $x=a$ $h_1(x)\rightarrow\tfrac12h_0(x)$ with probability one.
\end{lemma}
\begin{proof}Note that for each $x\in[\delta,a-\delta]$
\begin{align*}
&|h_1(x)-h_0(x)|=\tfrac12\lambda^{-1/2}\left|\int_0^ae^{-|x-y|/\sqrt{\l}}\hat{h}_n(y)\,dy-\int_{-\infty}^{\infty}e^{-|x-y|/\sqrt{\l}}h_0(x)\,dy\right|\\
&\,\,\,\le\sup_{[\delta,a-\delta]}|\hat{h}_n(z)-h_0(z)|+\tfrac12 \lambda^{-1/2}h_0(x)\left(\int_{-\infty}^0+\int_{a}^{\infty}\right)e^{-|x-y|/\sqrt{\l}}\,dy\\
&\,\,\,\le\sup_{[\delta,a-\delta]}|\hat{h}_n(z)-h_0(z)|+\tfrac12 h_0(a)\left(\int_{-\infty}^{-x/\sqrt{\lambda}}+\int_{(a-x)/\sqrt{\lambda}}^{\infty}\right)e^{-|v|}\,dv\\
&\,\,\,\le\sup_{[\delta,a-\delta]}|\hat{h}_n(z)-h_0(z)|+ h_0(a)\int_{\delta/\sqrt{\lambda}}^{\infty}e^{-v}\,dv\rightarrow0
\end{align*}
in probability as $n\rightarrow\infty$, where the upper bound is uniform in $x\in[\delta,a-\delta]$. Here we use Lemma \ref{lem:conshnint}.

Now consider $x=0$. We have
$$
h_1(0)=\tfrac12\lambda^{-1/2}\int_0^ae^{-y/\sqrt{\l}}\hat{h}_n(y)\,dy=\tfrac12\int_0^{a/\sqrt{\l}}\hat{h}_n(y\sqrt{\l})e^{-y}\,dy\rightarrow\tfrac12h_0(0)
$$
a.s.\ as $n\rightarrow\infty$. For $x=a$ the result follows analogously.\end{proof}

\noindent
In the lemma below, we investigate the asymptotics of  the constants $c_1$ and $c_2$ in Lemma \ref{lem:variation} as $\lambda=\lambda_n\downarrow0$.
\begin{lemma}
\label{lem:asympc}
Let $\hat h_n$ be the boundary-penalized least squares estimator of section \ref{section:isotonic_estimates}. Let $\a_n$ and $\b_n$ be of the order $n^{-2/3}$. Then, for $\lambda\downarrow0$,
$$
\bar c_1=\hat{h}_n(0)-h_1(0)+o_P(e^{-a/\sqrt{\lambda}}),\,\,\bar c_2=\hat{h}_n(a)-h_1(a)+o_P(e^{-a/\sqrt{\lambda}})
$$
and
\begin{eqnarray}
\check c_1&=& \int_0^{a/\sqrt{\l}} e^{-x}\hat h_n(x\sqrt{\l})\,dx+\label{eq:checkc1as}\\
&&\,\,\,\,\,-\int_{x=0}^{a/\sqrt{\l}}e^{-x}\int_{y=0}^x \hat h_n(y\sqrt{\l})e^{-(x-y)}\,dy\,dx+o_P(1)\mbox{ and }\nonumber \\
\check c_2&=& \int_0^{a/\sqrt{\l}} e^{-x}\hat h_n(a-x\sqrt{\l})\,dx+\label{eq:checkc2as}\\
&&\,\,\,\,\,-\int_{x=0}^{a/\sqrt{\l}}e^{-x}\int_{y=0}^x \hat h_n(a-y\sqrt{\l})e^{-(x-y)}\,dy\,dx+o_P(1).
\nonumber
\end{eqnarray}
Consequently, for $\l\downarrow0$ and under the conditions of Corollary \ref{lem:conshn}, $\check{c}_1,\bar{c}_1\stackrel{p}\longrightarrow\tfrac12 h_0(0)$ and $\check{c}_2,\bar{c}_2\stackrel{p}\longrightarrow\tfrac12 h_0(a)$
\end{lemma}
\begin{proof} For $\bar c_1$ and $\bar c_2$ the result immediately follows from (\ref{eq:c1_2}) and (\ref{eq:c2_2}) and Corollary \ref{lem:conshn}. For $\check c_1$ note that
\begin{align}
\label{eq:derivrel}
&\l^{-1/2}h_1(x)-h_1'(x)=\frac1{2\l}\int_0^a \hat h_n(y)e^{-|x-y|/\sqrt{\l}}\,dy
+\frac1{2\l}\int_0^x \hat h_n(y)e^{-(x-y)/\sqrt{\l}}\,dy\\
&\quad\quad\quad\quad-\frac1{2\l}\int_x^a \hat h_n(y) e^{(x-y)/\sqrt{\l}}\,dy
=\frac1{\l}\int_0^x \hat h_n(y)e^{-(x-y)/\sqrt{\l}}\,dy,\nonumber
\end{align}
implying (\ref{eq:checkc1as}).
Using that
$$
\l^{-1/2}h_1(x)+h_1'(x)=\frac1{\l}\int_x^a \hat h_n(y)e^{(x-y)/\sqrt{\l}}\,dy,
$$
(\ref{eq:checkc2as}) follows similarly. The last statements on the convergence in probability of the $c_i$'s use Lemma \ref{lem:unifsmallerset}. For $\check{c}_1$, note that the second term in (\ref{eq:checkc1as}) can be written as
$$
\int_0^{a/\sqrt{\lambda}}e^{-2x}\int_0^x\hat{h}_n(y\sqrt{\lambda})e^{y}\,dydx=\frac12\int_0^{a/\sqrt{\lambda}}\hat{h}_n(y\sqrt{\lambda})\left(e^{-y}-e^{y-2a/\sqrt{\lambda}}\right)\,dy.
$$
\end{proof}

Pictures for $n=100$ of the estimates of $\check{h}_n$, its derivative $\check{h}_n'$ and the density $\check f_n$ are shown in Figure \ref{fig:hazard2}, where $\l=0.10$. The same pictures for (the boundary-constrained) $\bar{h}_n$ are shown in Figure \ref{fig:hazard3}.

\begin{figure}
\begin{center}
\includegraphics[width=4cm]{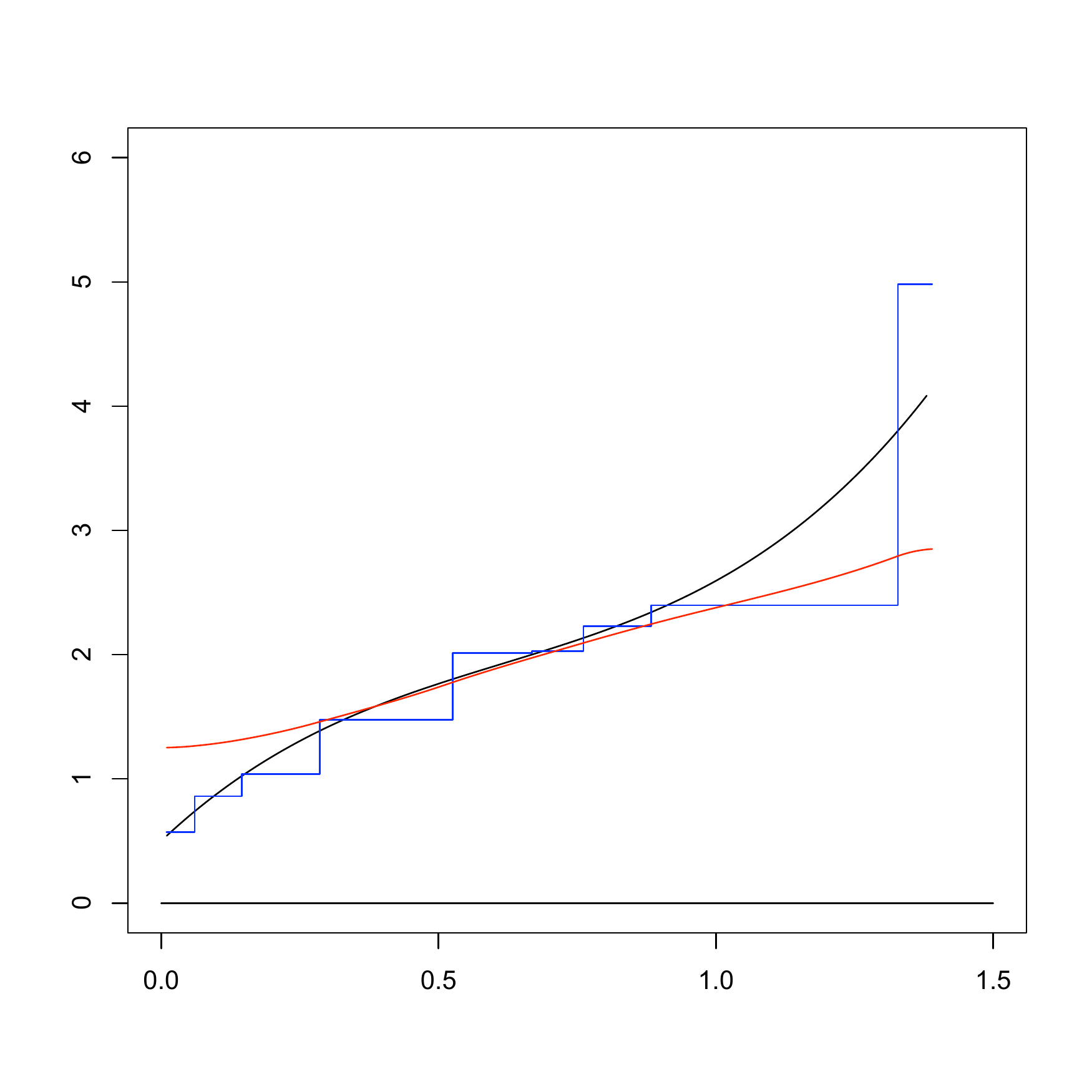}
\includegraphics[width=4cm]{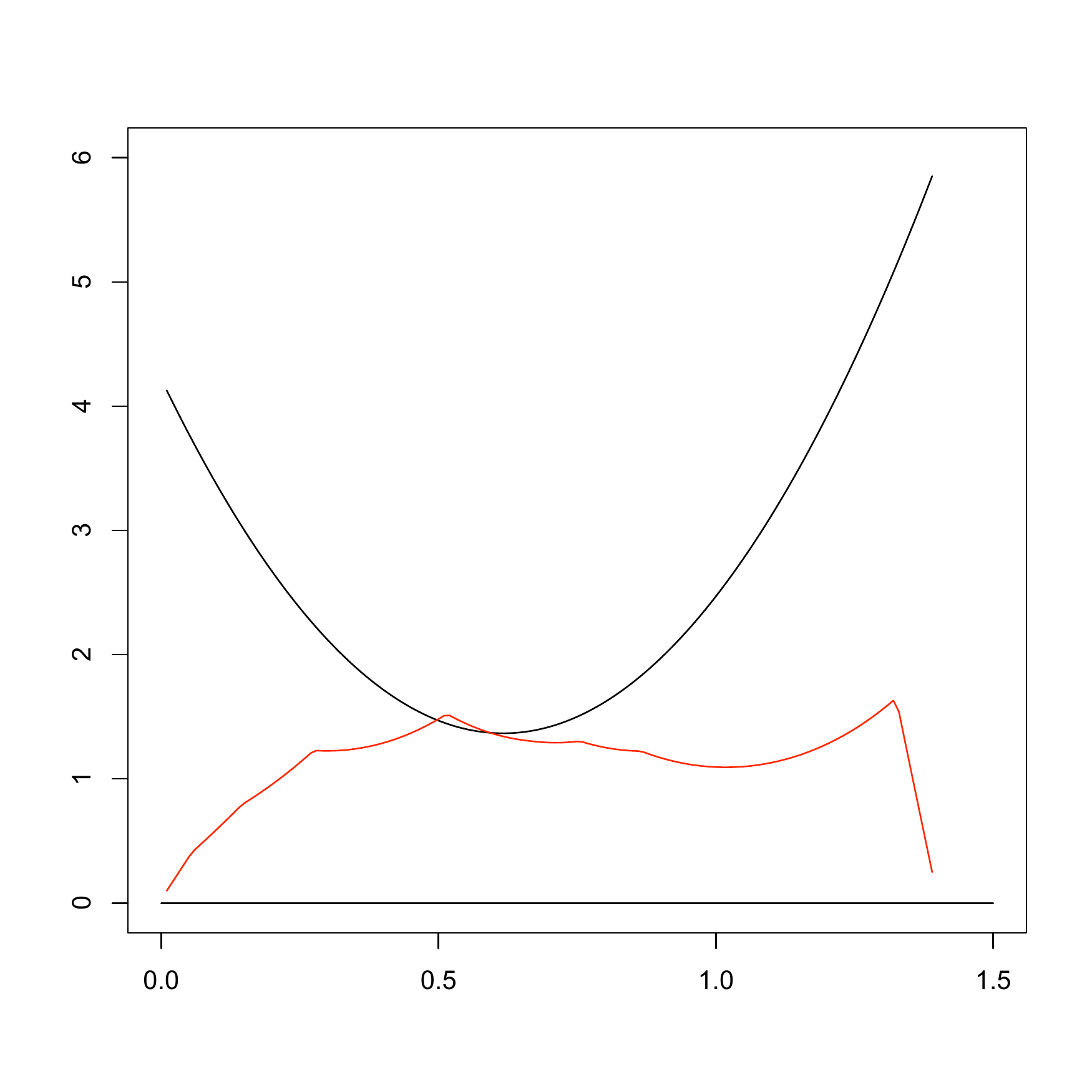}
\includegraphics[width=4cm]{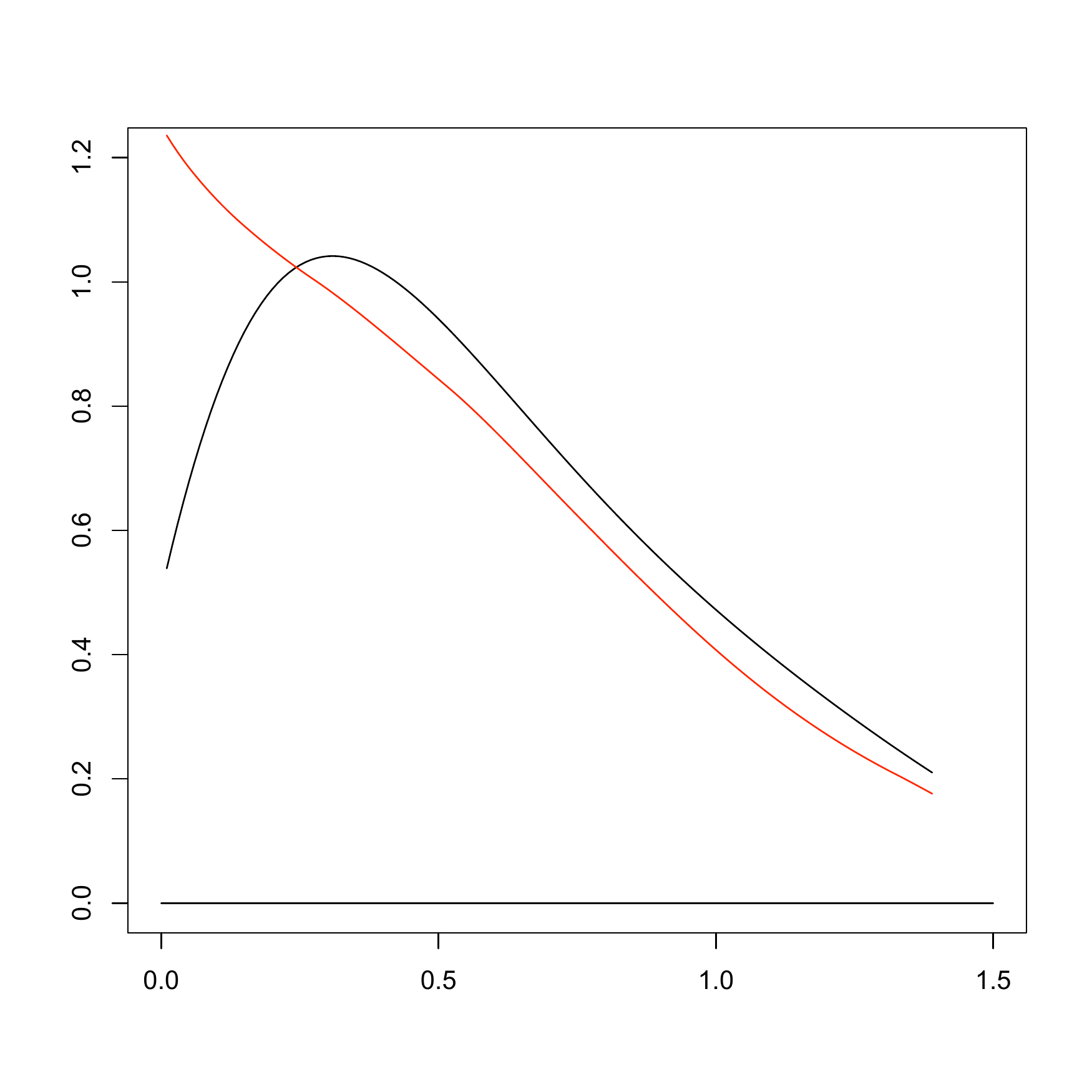}
\end{center}
\caption{The left panel shows the estimates $\hat h_n$ (blue), $\check h_n$ (red) of the hazard $h^{(1)}$ (of the family $\{h^{(d)}:d\in[-1,1]\}$) for a sample of size $n=100$, and the real density $h^{(1)}$ (black) on the $95\%$ percentile interval $[0,\left(F^{(1)}\right)^{-1}(0.95)]$; the corresponding derivatives of the hazard rates and the densities are shown in the middle- and right picture.}
\label{fig:hazard2}
\end{figure}

\begin{figure}
\begin{center}
\includegraphics[width=4cm]{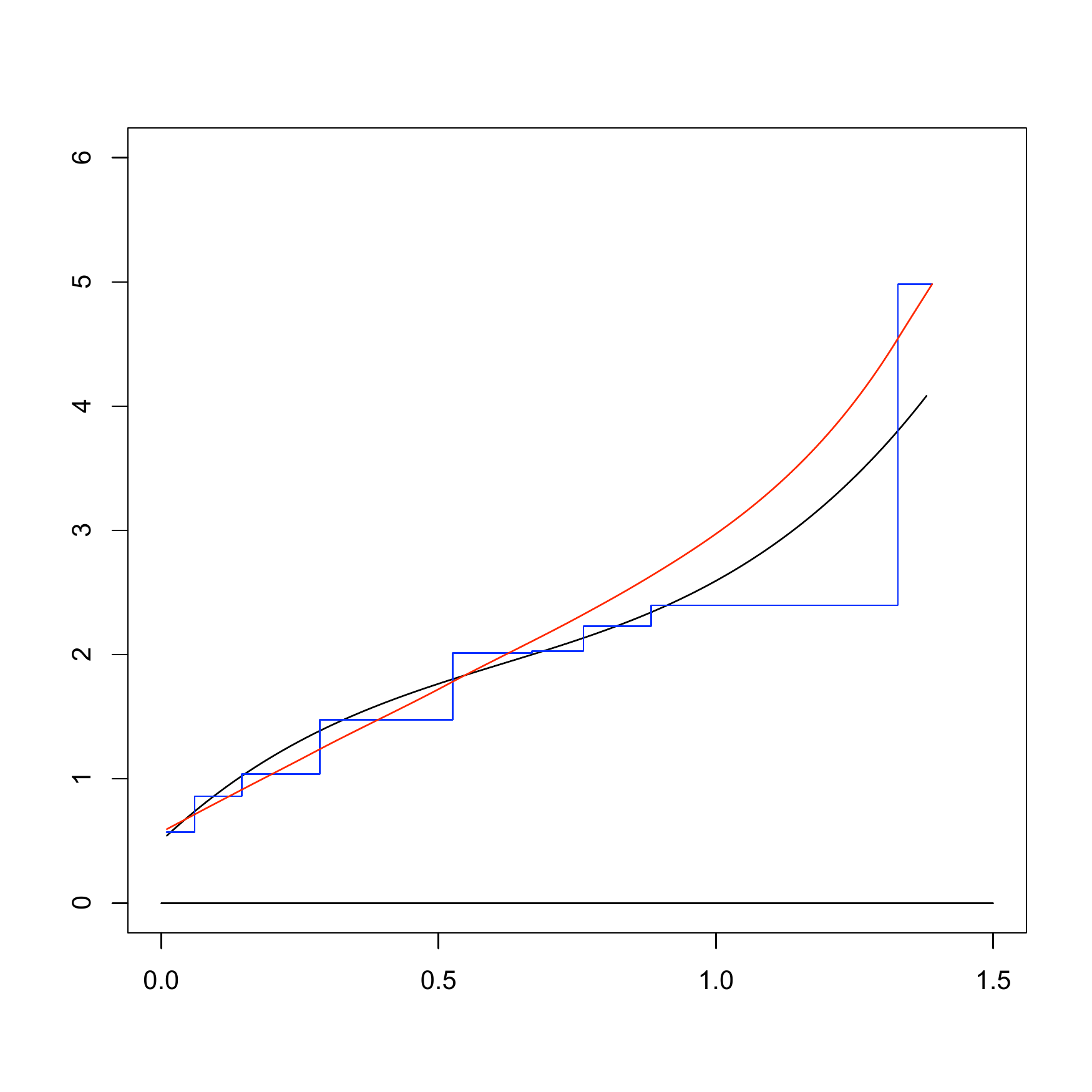}
\includegraphics[width=4cm]{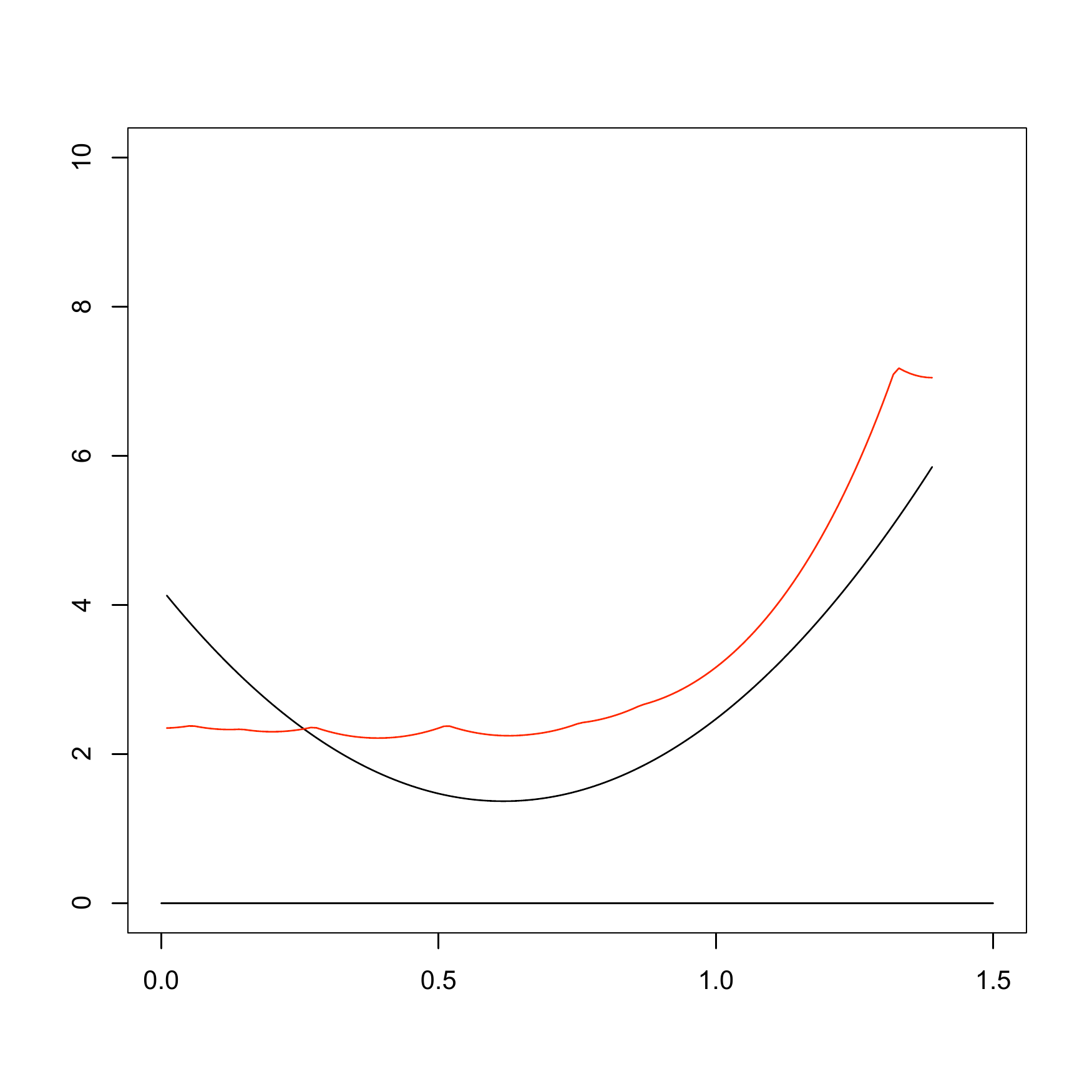}
\includegraphics[width=4cm]{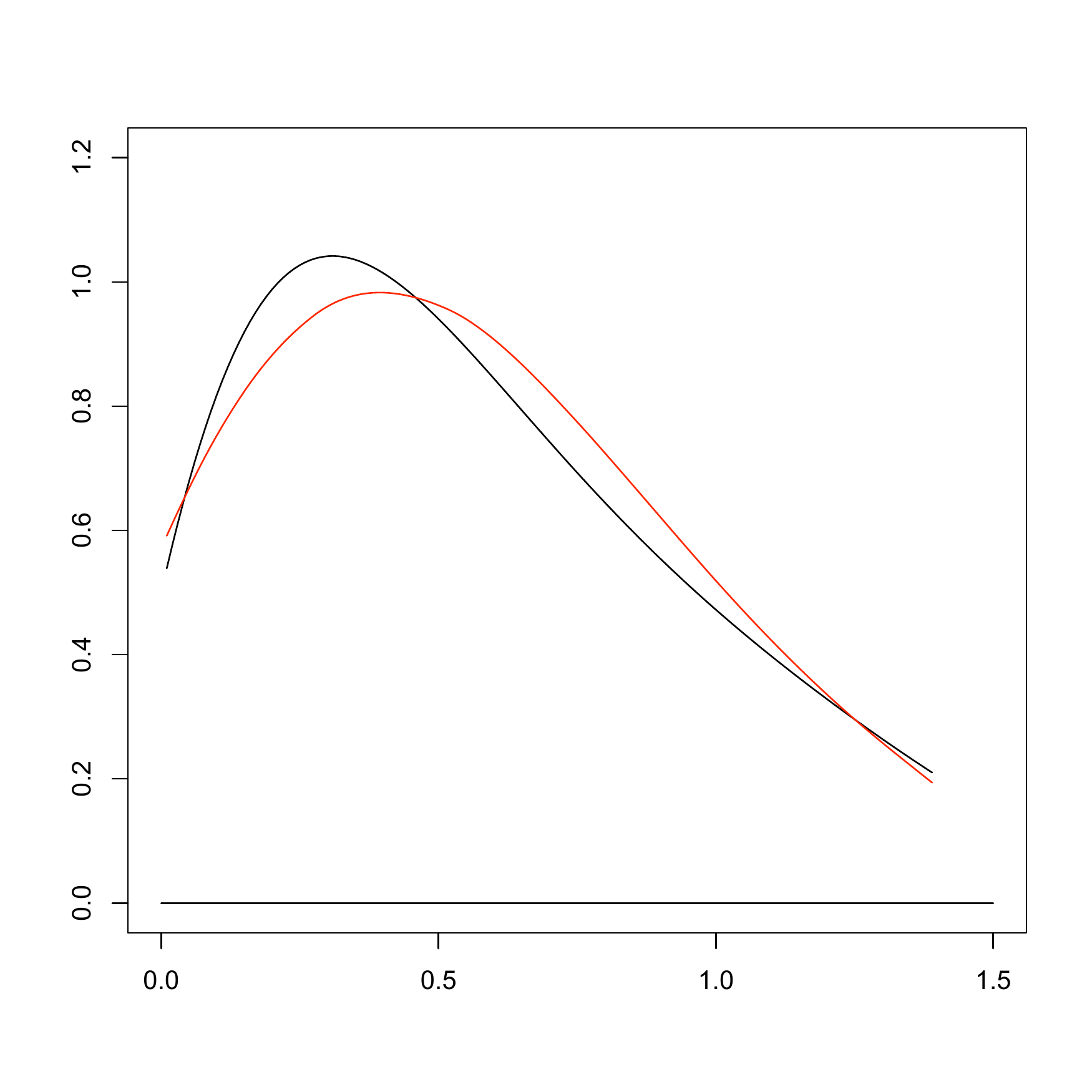}
\end{center}
\caption{The same pictures as in Figure \ref{fig:hazard2}, but with the boundary constrained $\bar{h}_n$ instead of $\check{h}_n$.}
\label{fig:hazard3}
\end{figure}

These pictures suggest that $\bar{h}$ behaves better than $\check{h}$. The following two results confirm this asymptotically. Theorem \ref{th:Wood1} shows that $\check h_n$ and $\bar h_n$ both estimate $h_0$ uniformly consistently. Theorem \ref{th:consistency_derivative} states that $\bar{h}_n^\prime$ does estimate $h_0^\prime$ consistently on the interval $[0,a]$, whereas $\check h_n^\prime$ is inconsistent at the boundaries $0$ and $a$.

\begin{thm}
\label{th:Wood1}
Let $\a_n,\b_n\asymp n^{-2/3}$, let  $\hat h_n$ be the nondecreasing minimizer of (\ref{LS_criterion2}), and let $\l=\l_n\to0$ as $n\to\infty$. Furthermore, assume that $h_0$ is continuously differentiable on $(0,a)$, with finite right and left limits at $0$ and $a$, respectively.
Then, if $\check h_n$ and $\bar h_n$  are the minimizers of Corollary \ref{lem:variation}, we have for each $x\in[0,a]$:
$$
\sup_{[0,a]}|\check h_n(x)-h_0(x)|\stackrel{p}\longrightarrow 0 \mbox{ and }\sup_{[0,a]}|\bar h_n(x)-h_0(x)|\stackrel{p}\longrightarrow 0,\,n\to\infty,
$$
\end{thm}

\begin{thm}
\label{th:consistency_derivative}
Let $\bar h_n$and $\check h_n$ be the boundary constrained minimizers of Corollary \ref{lem:variation}. Then, under the conditions of Theorem \ref{th:Wood1} and $n^{1/2}\l_n\rightarrow\infty$, we have for  each $0<\delta<a/2$ that
\begin{equation}
\label{eq:unifcheckh}
\sup_{[\delta,a-\delta]}|\check{h}_n^\prime (x)-h_0^\prime(x)|\stackrel{p}\longrightarrow 0 \mbox{ and } \sup_{[\delta,a-\delta]}|\bar h_n^\prime (x)-h_0^\prime(x)|\stackrel{p}\longrightarrow 0,\,n\to\infty.
\end{equation}
Moreover, $\check{h}_n^\prime(0)\stackrel{p}\longrightarrow 0$ and $\bar h_n^\prime(0)\stackrel{p}\longrightarrow h_0^\prime(0)$.
\end{thm}

\appendix
\section{Appendix Section}
\begin{lemma}
\label{lem:locsum}
Let $W_n$ be as defined in (\ref{eq:convbrown}) and assume the conditions of Theorem \ref{th:optimal_penalization}.
Consider the process
$$
[0,\nu n^{1/3})\ni t\mapsto \tilde{V}_n(t)=n^{1/3}V_n(t)=\frac{W_n(t)}{t}+\frac{\alpha}{t}+\tfrac12h_0^{\prime}(\xi_n)t
$$
where $\xi_n\in[0,\nu]$.
Then, by choosing $\epsilon>0$ sufficiently small and $M>0$ sufficiently large, for all large $n$
\begin{equation}
\label{eq:ineqVtilde}
\inf_{t\in [0,\epsilon]}\tilde{V}_n(t) > \tilde{V}_n(1) \mbox{ and }\inf_{t\in[M,\nu n^{1/3}]}\tilde{V}_n(t)> \tilde{V}_n(1)
\end{equation}
with probability arbitrarily close to one. This implies that with probability arbitrarily close to one for all large $n$
$$
\inf_{t\in[0,\nu n^{1/3}]}\tilde{V}_n(t)=\inf_{t\in(\epsilon,M]}\tilde{V}_n(t).
$$
\end{lemma}
\begin{proof} From (\ref{eq:convbrown}) it follows that $\tilde{V}_n(1)=O_P(1)$. Also from (\ref{eq:convbrown}), we have that for any $\epsilon>0$,
$$
\inf_{[0,\epsilon]}W_n(t)\Dconv  \inf_{[0,\epsilon]}W(h_0(0)t)=^\D -\sqrt{h_0(0)\epsilon}|Z|
$$
where $Z\sim N(0,1)$. Hence, with probability arbitrarily close to one,\\ $\inf_{[0,\epsilon]}W_n(t)>-\alpha/2$, implying
$$
|\tilde{V}_n(t)|\ge \frac{\alpha+\inf_{s\in[0,\epsilon]}W_n(s)}{\epsilon}\ge\frac{\alpha}{2\epsilon} \mbox{ for } t\in (0,\epsilon],
$$
proving the first statement in (\ref{eq:ineqVtilde}).

For the second statement, it suffices to show that for any $C$, the probability
$$
P\left(\inf_{[M,\nu n^{1/3}]}\tilde{V}_n(t)<C\right)\le\sum_{j=M}^{\nu n^{1/3}}P\left(\inf_{[j,j+1]}\tilde{V}_n(t)<C\right)
$$
can be made arbitrarily small by taking $M>0$ sufficiently large. Fix $C>0$ and take $M>0$ sufficiently large such that for all $t\ge M$, $(t+1)-\alpha/C-\tfrac14h_0^\prime(0)t^2<-\tfrac15h_0^\prime(0)t^2$. Then, using that $h_0^\prime(x)\ge \tfrac12h_0^\prime(0)$ on $[0,\nu]$ and taking $j\ge M$,
\begin{eqnarray*}
&&\inf_{[j,j+1]}\tilde{V}_n(t)<C\Rightarrow \exists t\in [j,j+1]\,:\,\tilde{V}_n(t)<C\\
&&\,\,\,\,\,\Rightarrow \exists t\in [j,j+1]\,:\,W_n(t)<Ct-\alpha-\tfrac14Ch_0^{\prime}(0)t^2\le-\tfrac15Ch_0^{\prime}(0)t^2\\
&&\,\,\,\,\,\Rightarrow \exists t\in [j,j+1]\,:\,n^{2/3}\left|\H_n(n^{-1/3}t)-H_0(n^{-1/3}t)\right|>\tfrac15Ch_0^{\prime}(0)j^2\\
&&\,\,\,\,\,\Rightarrow  \exists t\in [j,j+1]\,:\,n^{2/3}\left|\F_n(n^{-1/3}t)-F_0(n^{-1/3}t)\right|>\tfrac1{10}Ch_0^{\prime}(0)j^2
\end{eqnarray*}
where in the last implication we use that $|\log(1-u)-\log(1-v)|\le 2|u-v|$ for $0\le u,v\le 1/2$. Using Markov's inequality, we obtain for $j\ge M$
\begin{align*}
&\!P\left(\inf_{[j,j+1]}\!\!\tilde{V}_n(t)<C\right)\le P\left(\sup_{[j,j+1]}\!\!n^{2/3}\!\left|\F_n(n^{-1/3}t)-F_0(n^{-1/3}t)\right|>\frac{Ch_0^{\prime}(0)j^2}{10}\right)\\
&\quad\quad\quad\le \frac{n^{4/3}E\sup_{[j,j+1]}\left|\F_n(n^{-1/3}t)-F_0(n^{-1/3}t)\right|^2}{\left(\tfrac1{10}Ch_0^{\prime}(0)j^2\right)^2}\\
&\quad\quad\quad\le 100\frac{n^{4/3}E\sup_{[j,j+1]}\left|\F_n(n^{-1/3}t)-F_0(n^{-1/3}t)\right|^2}{C^2h_0^{\prime}(0)^2j^4}.
\end{align*}
By maximal inequality 3.1(ii) in \cite{KP:90}, the numerator in this expression is bounded by  $C^\prime(j+1)$, giving
$$
\!P\left(\inf_{[M,\nu n^{1/3}]}\!\!\tilde{V}_n(t)<C\right)\le\!\sum_{j=M}^{\nu n^{1/3}}\!P\!\left(\inf_{[j,j+1]}\!\tilde{V}_n(t)<C\right)\le
\frac{100C^\prime}{C^2 h_0^{\prime}(0)^2}\!\sum_{j=M}^{\infty}(j+1)j^{-4}
$$
which can be made arbitrarily small by taking $M$ sufficiently large.\end{proof}

\noindent
\textsc{Proof} of Theorem \ref{th:Wood1}.
For $x\in[0,a]$, we have for $h_n$ either $\check h_n$ or $\bar h_n$,
\begin{align*}
&\!\!|{h}_n(x)-h_0(x)|\!=\!\\&\,\,\,\left|\tfrac12\int_{-x/\sqrt{\lambda}}^{(a-x)/\sqrt{\lambda}}\!\hat{h}_n(x+\sqrt{\lambda}v)e^{-|v|}\,dv\!+\!{c}_1e^{-x/\sqrt{\lambda}}\!+\!
{c}_2e^{-(a-x)/\sqrt{\lambda}}\!-\!h_0(x)\right|\\
&\,\,\,\le\tfrac12\int_{-x/\sqrt{\lambda}}^{(a-x)/\sqrt{\lambda}}|\hat{h}_n(x+\sqrt{\lambda}v)-{h}_0(x+\sqrt{\lambda}v)|e^{-|v|}\,dv+\\
&\quad\quad+
\tfrac12\left|\int_{-x/\sqrt{\lambda}}^{(a-x)/\sqrt{\lambda}}(h_0(x+\sqrt{\lambda}v)-{h}_0(x))e^{-|v|}\,dv\right|+\\
&\quad\quad+\left|h_0(x)\left(\tfrac12\int_{-x/\sqrt{\lambda}}^{(a-x)/\sqrt{\lambda}}e^{-|v|}\,dv-1\right)+{c}_1e^{-x/\sqrt{\lambda}}+
{c}_2e^{-(a-x)/\sqrt{\lambda}}\right|\\&\,\,\,= I_n^{(1)}+I_n^{(2)}+I_n^{(3)}
\end{align*}
First, observe that by Corollary \ref{lem:conshn} and the assumed smoothness of $h_0$, for $n\rightarrow\infty$
$$
I_n^{(1)}\le\sup_{[0,a]}|\hat{h}_n(y)-h_0(y)|\stackrel{p}\longrightarrow0 \mbox{ and }I_n^{(2)}\le \tfrac12\sup_{[0,a]}|h_0^{\prime\prime}(y)|\l\int _0^{\infty} v^2e^{-v}\,dv\rightarrow0,
$$
where the upper bounds do not depend on $x$. Furthermore, note that
\begin{align*}
&I_n^{(3)}=\left|-\tfrac12h_0(x)\left(\int_{-\infty}^{-x/\sqrt{\lambda}}+\int_{(a-x)/\sqrt{\lambda}}^{\infty}\right)e^{-|v|}\,dv+{c}_1e^{-x/\sqrt{\lambda}}+
{c}_2e^{-(a-x)/\sqrt{\lambda}}\right|\\
&\,\,\,\,\,\,\,\,\le |{c}_1-\tfrac12h_0(x)|e^{-x/\sqrt{\lambda}}+|{c}_2-\tfrac12h_0(x)|e^{-(a-x)/\sqrt{\lambda}}.
\end{align*}
Therefore, for $0\le x\le \l^{1/4}\le a/2$, we have for
\begin{eqnarray*}
I_n^{(3)}&\le&|{c}_1-\tfrac12h_0(0)+\tfrac12(h_0(0)-h_0(x))|+(|{c}_2|+\tfrac12h_0(a))e^{-a/2\sqrt{\lambda}}\\
&\le&|{c}_1-\tfrac12h_0(0)|+\tfrac12\l^{1/4}\sup_{[0,a]}|h_0^\prime(y)| +(|{c}_2|+\tfrac12h_0(a))e^{-a/2\sqrt{\lambda}}\stackrel{p}\longrightarrow0
\end{eqnarray*}
by (\ref{eq:checkc1as}), where this upper bound is again independent of $x\in[0,\lambda^{1/4}]$. For $x\in[a-\lambda^{1/4},a]$ a similar argument yields an upper bound that does not depend on $x$ and converges to zero in probability. For $\lambda^{1/4}\le x\le a-\l^{1/4}$, we have
\begin{eqnarray*}
I_n^{(3)}&\le&(|{c}_1|+\tfrac12h_0(a))e^{-1/\lambda^{1/4}}+(|{c}_2|+\tfrac12h_0(a))e^{-1/\lambda^{1/4}}=O_P(e^{-1/\l^{1/4}}),
\end{eqnarray*}
again with an upper bound not depending on $x$. These inequalities, combined with Lemma \ref{lem:asympc} lead to (\ref{eq:unifcheckh}).\hfill$\Box$

\noindent
\textsc{Proof} of Theorem \ref{th:consistency_derivative}.
Using the expression for $h_1^\prime$ implicit in (\ref{eq:derivrel}), we get for $x\in(0,a)$
$$ h_1^\prime(x)=\frac1{2\sqrt{\l}}\left(\int_0^{(a-x)/\sqrt{\l}}\hat{h}_n(x+y\sqrt{\l})e^{-y}\,dy-\int_0^{x/\sqrt{\l}}\hat{h}_n(x-y\sqrt{\l})e^{-y}\,dy\right).
$$
Fix $0<\delta<a/2$. Using Corollary \ref{lem:conshn},
\begin{equation}
\label{eq:Vn1}
V_n=\sup_{x\in[0,a]}\left|\hat h_n(x)-h_0(x)\right|=O_P\left(n^{-1/4}\right),
\end{equation}
we can write (for $\delta<x\le a/2$; the situation $a/2\le x<a-\delta$ is similar)
\begin{align*}
&\!|h_1^\prime(x)\!-\!h_0^\prime(x)|\le\left|\int_0^{x/\sqrt{\l}}\!\left(\frac{h_0(x+y\sqrt{\l})\!-\!h_0(x-y\sqrt{\l})\!+\!2V_n}{2\sqrt{\lambda}}\!-\!h_0^\prime(x)\right)\!e^{-y}\,dy \right|\\
&\,\,+\!h_0^{\prime}(x)e^{-x/\sqrt{\l}}\!+\!\frac{1}{2\sqrt{\l}}\left| \int_{x/\sqrt{\l}}^{(a-x)/\sqrt{\l}}\!\!\hat{h}_n(x\!+\!y\sqrt{\l})e^{-y}dy \right|\le\! \frac{V_n}{\sqrt{\l}}\!+\!\sup_{[0,a]}h_0^\prime(y)e^{-\delta/\sqrt{\l}}\\
&\,\,+\frac{\hat{h}_n(a)}{2\sqrt{\l}}e^{-\delta/\sqrt{\l}}+\left|\int_0^{x/\sqrt{\l}}\left((y-1)h_0^\prime(x)+y^2\sqrt{\l}\sup_{[0,a]}|h_0^{\prime\prime}(z)|\right)\,e^{-y}\,dy \right|.
\end{align*}
The first three terms in the upper bound are $o_P(1)$, uniformly in $x$, where we use that $\lambda$ does not converge to zero too rapidly. The same holds for the last term, since it is bounded by
$$
\sup_{[0,a]}h_0^\prime(z)\int_{\delta/\sqrt{\l}}^{\infty}|y-1|\,e^{-y}\,dy +
\sqrt{\l}\sup_{[0,a]}|h_0^{\prime\prime}(z)|\int_0^{\infty}y^2\,e^{-y}\,dy.
$$
Also using Lemma \ref{lem:asympc}, this proves (\ref{eq:unifcheckh}).

Now consider the situation at zero. First for the estimator  $\check h_n$. Note that, using (\ref{eq:derivrel}) and Lemma \ref{lem:asympc},
\begin{align*}
&\check{h}'(0)= h_1'(0)-\frac{\check{c}_1}{\sqrt{\l}}+o_P(1)=
\frac1{2\l}\int_0^a \hat h_n(x) e^{-x/\sqrt{\l}}\,dx+o_P(1)+\\
&\,\,\,\frac{-1}{\sqrt{\l}}\int_0^{a/\sqrt{\l}} e^{-x}\hat h_n(x\sqrt{\l})\,dx
+\frac1{\sqrt{\l}}\int_{x=0}^{a/\sqrt{\l}}e^{-x}\int_{y=0}^x \hat h_n(y\sqrt{\l})e^{-(x-y)}\,dy\,dx\\
&=\frac{-1}{2\sqrt{\l}}\int_0^{a/\sqrt{\l}} \!\hat h_n\left(x\sqrt{\l}\right) e^{-x}\,dx\!+\!\frac1{\sqrt{\l}}\int_{x=0}^{a/\sqrt{\l}}\!\!e^{-x}\int_{y=0}^x \!\hat h_n\left(y\sqrt{\l}\right)e^{-(x-y)}\,dy\,dx\\&\,\,\,+o_P(1)
= -\tfrac12h_0'(0)+\tfrac12h_0'(0)+o_P(1)=o_P(1),\,n\to\infty,
\end{align*}
where we also use (\ref{eq:Vn1}) to obtain the last line. So we get:
$
\check h_n'(0)\stackrel{p}\longrightarrow0$ for $n\to\infty
$,
which means that $\check h_n'$ is inconsistent at zero. The other boundary point $a$ can be treated in a similar way.
Finally, consider the behavior of $\bar h_n^\prime$ at zero. Using (\ref{eq:derivrel}) and Lemma \ref{lem:asympc}, we get
\begin{align*}
&\bar{h}_n^\prime(0)=h_1^{\prime}(0)-\lambda^{-1/2}\bar{c}+o_P(1)=(h_1(0)-\bar{c})/\sqrt{\l}+o_P(1)\\&\,\,\,\,=(2h_1(0)\!-\!\hat{h}_n(0))/\sqrt{\l}\!+\!o_P(1)
\!=\!\l^{-1}\!\!\int_0^a\!\hat{h}_n(y)\,e^{-y/\sqrt{\l}}\,dy\!-\!\l^{-1/2}\hat{h}_n(0)\!+\!o_P(1)\\&\,\,\,\,=\int_0^{a/\sqrt{\l}}\frac{\hat{h}_n(y\sqrt{\l})-\hat{h}_n(0)}{\sqrt{\l}} e^{-y}\,dy+o_P(1).
\end{align*}
Using (\ref{eq:Vn1}), note that
\begin{align*}
&\left|\int_0^{a/\sqrt{\l}}\frac{\hat{h}_n(y\sqrt{\l})-\hat{h}_n(0)}{\sqrt{\l}} e^{-y}\,dy-h_0^\prime(0)\right|\le\\&\,\,\,\,
\left|\int_0^{a/\sqrt{\l}}\left(\frac{h_0(y\sqrt{\l})-h_0(0)}{\sqrt{\l}}-h_0^{\prime}(0)\right) e^{-y}\,dy\right|+\frac{V_n}{\sqrt{\l}}+o_P(1)=o_P(1),
\end{align*}
under the assumptions of our theorem.
\hfill$\Box$


\begin{thebibliography}{9}

\bibitem{chern:64}
\textsc{Chernoff, H.} (1964). Estimation of the Mode. \textit{The Annals of Statistical Mathematics} \textbf{16} 31--41.

\bibitem{durot:08} \textsc{Durot, C.}
(2008). Testing Convexity or Concavity  of a Cumulated Hazard Rate. \textit{IEEE Transactions on Reliability} \textbf{57} 465--473.

\bibitem{EJZ:98} \textsc{Es, A.J. van, Jongbloed, G.} and \textsc{Zuijlen, M.C.A. van} (1998). Isotonic inverse estimators for nonparametric deconvolution. \textit{The Annals of Statistics} \textbf{26}  2395--2406.

\bibitem{piet:83} \textsc{Groeneboom, P.} (1983).
The concave majorant of Brownian motion,
\textit{The Annals of Probability} \textbf{11} 1016--1027.

\bibitem{piet:89}\textsc{Groeneboom, P.} (1989).
Brownian motion with a parabolic drift and Airy functions.
\textit{Probability Theory and Related Fields} \textbf{ 81} 79--109.

\bibitem{GroJo95} \textsc{Groeneboom, P.} and \textsc{Jongbloed, G.}
(1995). Isotonic estimation and rates of convergence in Wicksell's problem.
\textit{The Annals of Statistics} \textbf{23} 1518--1542.

\bibitem{GrJo10a} \textsc{Groeneboom, P.}and \textsc{Jongbloed, G.}
(2010). Generalized continuous isotonic regression.
\textit{Statistics and Probability Letters} \textbf{80} 248--253.

\bibitem{GrJo10b} \textsc{Groeneboom, P.} and \textsc{Jongbloed, G.}
(2010). Testing monotonicity of a hazard: asymptotic distribution theory.
Submitted.

\bibitem{GrJo10c} \textsc{Groeneboom, P.} and \textsc{Jongbloed, G.}
(2010). Isotonic $L_2$-projection test for local monotonicity of a hazard.
Submitted.

\bibitem{grwe:01} \textsc{Groeneboom, P.} and \textsc{Wellner, J.A.} (2001). Computing Chernoff's distribution. \textit{Journal of Computational and Graphical Statistics} \textbf{10} 388--400.

\bibitem{gh04} \textsc{Gijbels, I.} and \textsc{Heckman, N.} (2004). Nonparametric testing for a monotone hazard function via normalized spacings. \textit{Journal of Nonparametric Statistics} \textbf{16}  463--478.


\bibitem{hallk:05} \textsc{Hall, P.} and \textsc{Keilegom, I.\ van}
(2005). Testing for Monotone Increasing Hazard Rate. \textit{The Annals of Statistics} \textbf{33} 1109--1137.

\bibitem{kiefwolf:76} \textsc{Kiefer, J.} and \textsc{Wolfowitz, J.} (1976). Asymptotically minimax estimation of concave and convex distribution functions. \textit{Zeitschrift f\"ur Wahrscheinlichkeitstheorie und verwandte Gebiete} \textbf{32} 111--131.

\bibitem{KP:90} \textsc{Kim, J.} and \textsc{Pollard, D.} (1990). Cube root asymptotics. \textit{The Annals of Statistics} \textbf{18} 191--219.

\bibitem{neumeyer07} \textsc{Neumeyer, N.} (2007). A note on uniform consistency of monotone function estimators. \textit{Statistics and Probability Letters} \textbf{77} 693--703.

\bibitem{palwood:06} \textsc{Pal, J.K.} and \textsc{Woodroofe, M.}
(2006). On the distance between cumulative sum diagram and its greatest convex minorant for unequally spaced design points. \textit{ Scandinavian Journal of Statistics} \textbf{33} 279--291.

\bibitem{palwood:07} \textsc{Pal, J.K.} and \textsc{Woodroofe, M.}
(2007). Large sample properties of shape restricted regression estimators with smoothness adjustments. \textit{Statistica Sinica} \textbf{17} 1601--1616.

\bibitem{prakasarao:70} \textsc{Prakasa Rao, B.L.S.} (1970).
Estimation for distributions with monotone failure rate. \textit{The Annals of Mathematical Statistics} \textbf{41}, 507--519.

\bibitem{prospyke:67} \textsc{Proschan, F.} and \textsc{Pyke, R.}
(1967). Tests for monotone failure rate. \textit{Proceedings of the Fifth Berkeley Symposium on Mathematical Statistics and Probability} \textbf{3} 293--312.

\bibitem{ricerosen76} \textsc{Rice, J.} and \textsc{Rosenblatt, M.} (1976). Estimation of the log survivor function and hazard function. \textit{Sankhya A} \textbf{38} 60--78.

\bibitem{Robertson:88} \textsc{Robertson,
T., Wright, F.} and \textsc{Dykstra, R.}
(1972). \textit{Order Restricted Inference}. John Wiley \& Sons. New York.

\bibitem{singpurwallawong:83} \textsc{Singpurwalla, N.D.} and \textsc{Wong, M.Y.} (1983). Estimation of the failure rate - a survey of nonparametric methods, Part 1: Non-Bayesian methods. \textit{Communications in Statistics} \textbf{12} 559--588.

\bibitem{tantwood:94} \textsc{Tantiyaswasdikul, C.} and \textsc{Woodroofe, M.B.}
(1994). Isotonic smoothing splines under sequential designs.  \textit{Journal of Statistical Planning and Inference} \textbf{38} 75--88.

\bibitem{woodroofe_sun:93}
\textsc{Woodroofe, M.} and \textsc{Sun, J.} (1993). A penalized likelihood estimate of $f(0+)$
when $f$ is nonincreasing. \textit{Statistica Sinica} \textbf{3} 501--515.

\bibitem{woodroofe_sun:99}
\textsc{Woodroofe, M.}  and \textsc{Sun, J.}
(1999). Testing uniformity versus a monotone density. \textit{The Annals of Statistics} \textbf{27} 338--360.

\end{thebibliography}
\end{document}